\documentclass[12pt]{amsart}

\usepackage{amssymb, graphicx}

\usepackage{enumerate}
\usepackage{tikz}
           
\usetikzlibrary{angles}
\usetikzlibrary{decorations.markings,calc}
\usetikzlibrary{chains,positioning,quotes}
\usetikzlibrary{cd}
\usepackage{wasysym}
\usepackage[all]{xy}

\entrymodifiers={!!<0pt,0.7ex>+}  

\newtheorem{Lemma}             {Lemma}
\newtheorem{Corollary}  [Lemma]{Corollary}
\newtheorem{Proposition}[Lemma]{Proposition}
\newtheorem{Remark}     [Lemma]{Remark}

\newtheorem{Example}    [Lemma]{Example}
\newtheorem{Theorem}    [Lemma]{Theorem}

\theoremstyle{definition}
\newtheorem{Notation}[Lemma]{Notation}
\newtheorem{Rem}[Lemma]{Remark}

\newcommand{\op}{\operatorname}
\newcommand{\ov}{\overline}

\newcommand{\Z}{\mathbb Z}

\newcommand{\C}{\mathbb C}

\DeclareMathOperator\hd{hd}
\DeclareMathOperator\soc{soc}
\DeclareMathOperator\rad{rad}
\title{The self-dual indecomposable modules in blocks with cyclic defect groups}
\author{Caroline Lassueur}
\address{{Caroline Lassueur}, RPTU Kaiserslautern-Landau, Fachbereich Mathematik, 67653 Kaiserslautern, Germany and Leibniz Universit\"at Hannover, Institut f\"ur Algebra, Zahlentheorie und Diskrete Mathematik, Welfengarten 1, 30167 Hannover, Germany}
\author{John Murray} \address{{John Murray}, Department of Mathematics \& Statistics, Maynooth University, Ireland}
\date{\today}

\begin{document}
\begin{abstract}
Let\/ $p$ be an odd prime and let\/ $\mathbf{B}$ be a $p$-block of a finite group, such that\/ $\mathbf{B}$ has cyclic defect groups. We describe the self-dual indecomposable $\mathbf{B}$-modules and  for each such module determine whether it is symplectic or orthogonal.
\end{abstract}

\maketitle

\section{Introduction and statement of results}\label{sec:intro}%

Let\/ $G$ be a finite group and let\/ $k$ be an algebraically closed field of prime characteristic $p$. The ordinary irreducible characters $\op{Irr}(G)$ of\/ $G$ are partitioned into ($p$-)blocks, depending on their values on the $p'$-order elements of\/ $G$. These blocks are in bijection with the blocks of the group algebra $kG$, that is, the minimal two-sided ideals of\/ $kG$. Moreover, if\/ $M$ is an indecomposable $kG$-module, there is exactly one block of\/ $kG$ which does not annihilate~$M$. We then say that\/ $M$ belongs to the corresponding block of\/ $G$.

Let\/ $\mathbf{B}$ be a block of\/ $G$. The defect groups of\/ $\mathbf{B}$ are a certain conjugation orbit of\/ $p$-subgroups of\/ $G$. For example the block of\/ $G$ containing the trivial $kG$-module is called the principal block of\/ $G$ and its defect groups are the Sylow $p$-subgroups of\/ $G$. At the other extreme, any block which contains a single irreducible character of\/ $G$ has a trivial defect group.

Suppose that\/ $p\ne2$, that\/ $\mathbf{B}$ has a non-trivial cyclic defect group $D$ and $e$ irreducible modules. Then $e$ is called the inertial index of\/ $\mathbf{B}$, and it is known that\/ $\mathbf{B}$ has a finite number, namely $e|D|$, of isomorphism classes of indecomposable modules.

The Brauer tree of\/ $\mathbf{B}$ is a planar embedded graph $\sigma(\mathbf{B})$ which has $e$ edges labelled by the irreducible $\mathbf{B}$-modules. Now $\sigma(\mathbf{B})$ has $e+1$ vertices. One vertex is said to be \emph{exceptional}, and it comes with a positive integer $m:=(|D|-1)/e$, called its multiplicity (or the exceptional multiplicity of\/ $\sigma(\mathbf{B})$). The exceptional vertex is associated with a family of\/ $m$ {\em exceptional characters} in $\op{Irr}(\mathbf{B})$. The remaining $e$ vertices are then said to be non-exceptional, and each of them corresponds to a single non-exceptional character in $\op{Irr}(\mathbf{B})$. Thus $\op{Irr}(\mathbf{B})$ contains $e+m$ characters.  Notice that if\/ $m=1$, then any vertex may be designated as the exceptional vertex. Furthermore, it is known that the Morita equivalence class of\/ $\mathbf{B}$ is determined by $\sigma(\mathbf{B})$ together with its exceptional vertex and  exceptional multiplicity $m$. We make this more precise below.

The complex conjugates of the characters in $\mathbf{B}$ form another block\/ $\mathbf{B}^\circ$, called the contragredient of\/ $\mathbf{B}$. Moreover if\/ $M$ is an indecomposable $\mathbf{B}$-module then the dual $kG$-module $M^*:=\op{Hom}_k(M,k)$ is an indecomposable $\mathbf{B}^\circ$-module. We say that\/ $\mathbf{B}$ is a {\em real block} if it coincides with $\mathbf{B}^\circ$. 

Now suppose that\/ $\mathbf{B}$ is a real block. In \cite{Gr}, J. Green showed that the real valued non-exceptional characters in $\mathbf{B}$, together with the exceptional characters in $\mathbf{B}$, correspond to a line segment in $\sigma(\mathbf{B})$, known as the {\em real stem} of\/ $\sigma(\mathbf{B})$. The edges in the real stem correspond to the self-dual irreducible $\mathbf{B}$-modules. All or none of the exceptional characters are real-valued. Moreover, the complex conjugate of an exceptional character is exceptional and duality acts as reflection in the real stem. 

As $p$ is an odd prime, \cite[Proposition 2.1]{W} shows that each self-dual indecomposable $kG$-module $M$ affords a non-degenerate $G$-equivariant quadratic form ($+$ type) or a non-degenerate $G$-equivariant symplectic form ($-$ type). The type indicates whether the associated group representation $G\rightarrow\op{GL}(M)$ factors through an orthogonal subgroup $\op{O}(M)$ or a symplectic subgroup $\op{Sp}(M)$ of\/ $\op{GL}(M)$, respectively. 

Our aim in this article is two-fold. First, we want to give a concrete classification of the self-dual indecomposable $\mathbf{B}$-modules (up to isomorphism). Second, we want to determine their type. We will make use of Janusz' classification in \cite{Ja} of all indecomposable $\mathbf{B}$-modules in terms of certain walks on the Brauer tree called \emph{paths}. We then reduce the determination of the type of a self-dual module, to the determination of the type a certain self-dual indecomposable module of a local subgroup of\/ $G$. This in turn can be computed using the character table of the subgroup.

We summarize our main results in the following two theorems. The statements will be made more precise in Theorem \ref{thm:description_selfduals} and Theorem \ref{thm:type_selfduals_general_case} in Section \ref{sec:generalcase}:

\begin{Theorem}\label{thm:intro}%
Let\/ $M$ be a non-projective indecomposable module belonging to $\mathbf{B}$. Then $M$ is self-dual if and only if one of the following conditions is satisfied:
\begin{enumerate}[\rm (1)]
\item $M$ is an irreducible module indexed by an edge on the real stem of\/ $\sigma(\mathbf{B})$; or
\item $M$ is reducible and its Janusz path – in the sense of Notation~\ref{nota:paths_dir_mult} –  is fixed under reflection in the real stem.
\end{enumerate}
\end{Theorem}

We highlight a family of indecomposable  $\mathbf{B}$-modules, called \emph{hooks} in \cite[\S 2.3]{BC}, and defined as quotients of the PIMs of\/ $\mathbf{B}$ as follows. If\/ $E$ is an irreducible $\mathbf{B}$-module, then its projective cover $P(E)$ has $E=\soc(P(E))=\text{hd}(P(E))$. Moreover there are uniserial (possibly zero) $\mathbf{B}$-modules $Q_a$ and\/ $Q_b$ such that\/ $\rad(P(E))/\soc(P(E))=Q_a\oplus Q_b$. We can visualize $P(E)$  using the following diagram:\vspace{-1mm}
\[
P(E)=\boxed{\begin{smallmatrix} E \\ Q_a\oplus\, Q_b \\ E\end{smallmatrix}}\,.
\]
Then  the two uniserial quotients of\/ $P(E)$ of the form 
$$
H_a:=\boxed{\begin{smallmatrix}E\\Q_a\end{smallmatrix}}\qquad\text{ and }\qquad H_b:=\boxed{\begin{smallmatrix}E\\Q_b\end{smallmatrix}}
$$
are called \textit{hooks}. When $Q_a=0$ (resp. $Q_b=0$), then $H_a=E$  (resp. $H_b=E$) is irreducible. Clearly, $\mathbf{B}$ has exactly $2e$ isomorphism classes of hooks. We also note, that the hooks are precisely the modules occurring in Green's walk around the Brauer tree \cite{Gr}, and they form the boundary of the stable Auslander-Reiten quiver $\Gamma_s(\mathbf{B})$ of\/ $\mathbf{B}$. (See Subsection~\ref{ssec:sARquiver} for this notion.)

Regarding the type of the\/ $\mathbf{B}$-modules, we have:

\begin{Theorem}
Up to isomorphism $\mathbf{B}$ has exactly $|D|-1$ non-projective indecomposable self-dual modules. There are two self-dual hooks in $\mathbf{B}$.
\begin{enumerate}[\rm (1)]
\item Each self-dual indecomposable $\mathbf{B}$-module $M$ has the same type as a self-dual hook\/ $H$, with $H$ determined by the position of\/ $M$ in the stable Auslander-Reiten quiver $\Gamma_s(\mathbf{B})$.
\item The type of\/ $H$ coincides with the type of its Green correspondent\/ $g(H)$ in $\op{N}_G(D)$. Moreover the type of\/ $g(H)$ can be computed from the character table of\/ $\op{N}_G(D)$.
\end{enumerate}
\end{Theorem}

The paper is structured as follows. In Section \ref{sec:prelim} we set up our notation and prove preliminary results that allow us to describe the number of self-dual indecomposable modules in $\mathbf{B}$.  In Section \ref{sec:uniserial} we investigate our main problem for the case of a uniserial block, where we explicitly describe the type of each self-dual module. In Section \ref{sec:generalcase} we give a concrete classification of the self-dual indecomposable modules in $\mathbf{B}$ in terms of  Janusz' classification of the indecomposable $\mathbf{B}$-modules. Moreover, we give a procedure to determine the type of any such module. Finally, in Appendix~\ref{app:A}, for completeness, we briefly recap Janusz' parametrization  mentioned above.

\section{Notation and preliminary results}\label{sec:prelim}%

We begin with a condensed summary of some standard notation for blocks with cyclic defect groups which will be used throughout this article. All results mentioned below without reference can be found in the foundational paper of E.~Dade \cite{D} or \cite[\S6.5]{B}.

\subsection{General notation}

Let\/ $G$ be a finite group and let\/ $k$ be an algebraically closed field of prime characteristic $p\ne 2$. Moreover, let\/ $\mathbf{B}$ be a real $p$-block of\/ $G$ and let\/ $D$ be a defect group of\/ $\mathbf{B}$. We assume that\/ $D$ is cyclic of order $|D|=p^a$, with $a\geq 1$.

Let\/ $D_i$ be the unique subgroup of\/ $D$ which has order $p^i$, set\/ $N_i:=\op{N}_G(D_i)$ and let\/ $\mathbf{B}_i$ be a $p$-block of\/ $N_i$ such that\/ $\mathbf{B}_i^G=\mathbf{B}$, for each $i=0,\dots,a$. Let\/ $\mathbf{b}$ be a block of\/ $\op{C}_G(D)$ which has defect group $D$ and which satisfies $\mathbf{b}^G=\mathbf{B}$. If\/ $E$ is the stabilizer of\/ $\mathbf{b}$ in $N_G(D)$, then $e:=[E:\op{C}_G(D)]$ gives the inertial index of\/ $\mathbf{B}$.

Consider the action of\/ $E$ on $\op{Irr}(D)$. We choose a set\/ $\Lambda$ of representatives for the orbits of\/ $E$ on the non-trivial characters in $\op{Irr}(D)$. Each such orbit has cardinality $e$. So $|\Lambda|=\frac{|D|-1}{e}$. This number is the multiplicity $m$ mentioned in the introduction.

Recall that\/ $\mathbf{B}$ has $e$ isomorphism classes of irreducible modules, with representatives $E_0,...,E_{e-1}$,  and\/ $e+m$ ordinary irreducible characters, of which $e$ are non-exceptional and\/ $m$ are exceptional. So we use the notation
$$
\op{Irr}(\mathbf{B})=\{\chi_1,\dots,\chi_e\}\cup\{\chi_\lambda\mid\lambda\in\Lambda\},
$$
and set\/ $\chi_\Lambda:=\sum_{\lambda\in\Lambda}\chi_\lambda$ as the sum of the exceptional characters. 

Suppose now that\/ $m=1$. Then there is no canonical choice of exceptional character in~$\mathbf{B}$. (One character can artificially be designated to be exceptional.) Moreover, $E$ has a single orbit on the non-trivial characters in $\op{Irr}(D)$. So $|D|=p$, and there are $p$ characters in $\op{Irr}(\mathbf{B})$. As $p$ is odd, at least one irreducible character in $\mathbf{B}$ is real. So we can and do assume that\/ $\chi_\Lambda$ is real valued in this case.

Recall that\/ $\sigma(\mathbf{B})$ denotes the Brauer tree of\/ $\mathbf{B}$. For convenience we identify $\sigma(\mathbf{B})$ with a connected subset of\/ $\mathbb{C}$. The real irreducible characters are represented by vertices placed on $\mathbb{R}$, with the exceptional vertex placed at\/ $0$. So every self-dual irreducible $\mathbf{B}$-module is represented by an edge in $\mathbb R$. Moreover the complex conjugate of an irreducible character placed at\/ $x+iy$ is placed at\/ $x-iy$. In this way duality on $\sigma(\mathbf{B})$ extends to complex conjugation on $\mathbb{C}$. 
In our drawings we indicate the non-exceptional vertices with hollow nodes and the exceptional vertex with a black node.

It is known that\/ $\mathbf{B}$ has $e|D|$ isomorphism classes of indecomposable modules. We next outline two quite different ways of parametrizing the non-projective indecomposable $\mathbf{B}$-modules. In our discussion we may identify a module with its isomorphism class.

\subsection{Janusz' parametrization of the indecomposable modules}

First, it is convenient to split the isomorphism classes of indecomposable $\mathbf{B}$-modules into three families:
\begin{enumerate}[\hspace{4mm}\rm 1.]
\item  the $e$ isomorphism classes of irreducible $\mathbf{B}$-modules $E_0,...,E_{e-1}$;
\item  the $e$ isomorphism classes of projective indecomposable modules $P(E_0),...,P(E_{e-1})$, which are the projective covers of\/ $E_0,...,E_{e-1}$; and
\item  the $e(|D|-2)$ isomorphism classes of indecomposable $\mathbf{B}$-modules which are neither irreducible nor projective. 
\end{enumerate}

The Loewy structure of the projective indecomposable modules is well-known. We refer to \cite[\S4.18]{B} for a detailed description. Then, G. Janusz \cite{Ja} parametrized the $e(|D|-2)$ non-projective and non-irreducible indecomposable $\mathbf{B}$-modules using certain walks in the Brauer tree $\sigma(\mathbf{B})$, which he called \emph{paths}. We refer the reader to Appendix~\ref{app:A}, Notation \ref{nota:paths_dir_mult} for a more detailed description of this parametrization. In brief, each such indecomposable $\mathbf{B}$-module $X$ can be described by three parameters:
\begin{enumerate}[\hspace{4mm}\rm (1)]   
\item  a \emph{path} on $\sigma(\mathbf{B})$, which is a directed connected subgraph of {Type I} or {Type II} as defined in Notation \ref{nota:paths_dir_mult}, and the edges of which  form an ordered sequence $(S_1,\ldots,S_{t})$ of irreducible $\mathbf{B}$-modules such that the odd-indexed modules are in the head of\/ $X$  and the even-indexed edges are in the socle of\/ $X$, or conversely;
\item a \emph{direction} $(\varepsilon_1,\varepsilon_t)$, where for $i\in\{1,t\}$  we set\/ $\varepsilon_i=:1$ if\/ $S_i$ is in the head of\/ $X$ and\/ $\varepsilon_i=:-1$ if\/ $S_i$ is in the socle of\/ $X$;
\item a \emph{multiplicity} $\mu$, defined as follows. Set\/ $\mu:=0$ if\/ $m=1$, and if\/ $m>1$, then $\mu$ corresponds to the number of times that an irreducible module adjacent to the exceptional vertex on the path occurs as a composition factor of~$X$.
\end{enumerate}

Note in particular that a non-irreducible uniserial $\mathbf{B}$-module $M$ is described by a path of length~$2$ of the form
\[
\xymatrix@R=0.0000pt@C=30pt{
    {\overset{\chi_a}{\bullet}}\ar@<0.0ex>[r]^{T} &{\overset{\chi_b}{\bullet}}\ar@<0.0ex>[r]^{S} &{\overset{\chi_c}{\bullet}}
    }
\]
where the vertices $\chi_a,\chi_b,\chi_c$ belong to $\{\chi_1,\ldots,\chi_e\}\cup\{\chi_{\Lambda}\}$ and\/ $S$ and\/ $T$ are irreducible $\mathbf{B}$-modules. If the direction is $\varepsilon=(1,-1)$, then $T=\text{hd}(M)$ and\/ $S=\soc(M)$.

The hooks of\/ $\mathbf{B}$, which we defined in the introduction, are either irreducible, or uniserial with composition length greater than $1$, thus described by paths of length $2$, as given above. A precise description of all uniserial and all self-dual uniserial $\mathbf{B}$-modules in terms of path, direction, multipliciy is given in Remark \ref{rem:uniserials}(d) and  Lemma~\ref{lem:uniserial_selfdual}.

In Section \ref{sec:generalcase} we will describe all the self-dual indecomposable $\mathbf{B}$-modules in terms of these three parameters.

\subsection{The stable Auslander-Reiten quiver}\label{ssec:sARquiver}

Next, the Auslander-Reiten quiver $\Gamma_s(\mathbf{B})$ of\/ $\mathbf{B}$ is a directed graph whose vertices index the isomorphism classes of the  non-projective indecomposable $\mathbf{B}$-modules. This graph has a particularly rich structure.

Recall that the {\em projective cover} of a $kG$-module $M$ is a certain projective module $P(M)$, together with a short exact sequence of\/ $kG$-modules
$$
0\rightarrow\Omega^1(M)\rightarrow P(M)\rightarrow M\rightarrow 0.
$$
The kernel $\Omega^1(M)$ is called the {\em Heller translate} of\/ $M$.  Dualising the above short exact sequence gives another short exact sequence
$$
0\rightarrow M^*\rightarrow P(M)^*\rightarrow \Omega^1(M)^*\rightarrow 0.
$$
Now the {\em injective envelope of\/ $M^*$} is a certain injective $kG$-module $I(M^*)$, together with a short exact sequence of\/ $kG$-modules
$$
0\rightarrow M^*\rightarrow I(M^*)\rightarrow \Omega^{-1}(M^*)\rightarrow 0.
$$
But\/ $P(M)^*\cong I(M^*)$. So 
\begin{equation}\label{E:dual-omega}
\Omega^1(M)^*\cong\Omega^{-1}(M^*)\,,
\end{equation}
whence
\begin{equation}\label{E:dual-tau}
\Omega^2(M)^*\cong\Omega^{-2}(M^*).
\end{equation}

For a precise description of $\Gamma_s(\mathbf{B})$, we refer the reader to \cite[Section~6.5]{B}. 
Summed up, it is well-known, that in blocks with cyclic defect groups, all modules are periodic of period\/ $2e$.  It follows that the quiver $\Gamma_s(\mathbf{B})$ is a finite tube of type $(\Z/e\Z)A_{|D|-1}${, where the  Auslander-Reiten translate is   $\Omega^{-2}$}. This tube consists of\/ $|D|-1=em$ {\em levels} arranged consecutively from level $1$ to level $em$, with each level consisting of\/ $e$ indecomposable modules arranged in a loop. More precisely, level $\ell$ can be described as\/ $\{\Omega^{-2j}(M)\mid 0\leq j\leq e-1\}$  where $M$ is an arbitrarily chosen module on the level, and for each $j=0,\ldots,e-1$ the right successor of\/ $\Omega^{-2j}(M)$ on the loop is\/  $\Omega^{-2j-2}(M)$. For this reason the levels of\/ $\Gamma_s(\mathbf{B})$ are usually called \emph{$\Omega^2$-orbits}.

The $2e$ modules which constitute the two ends of the tube $(\Z/e\Z)A_{|D|-1}$ are the hooks of\/ $\mathbf{B}$. Moreover, there exists an irreducible $\mathbf{B}$-module $X$ which is a leaf of\/ $\sigma(\mathbf{B})$ such that level $1$ is $ \{\Omega^{-2j}(X)\mid 0\leq j\leq e-1\}$ and level $em$ is $ \{\Omega^{-2j+1}(X)\mid 0\leq j\leq e-1\}${, or conversely}.

We make a preliminary observation:

\begin{Lemma}\label{lem:SD_levels_ARquiver}
If\/ $e$ is odd, then each level of\/ $\Gamma_s(\mathbf{B})$ has $1$ self-dual module, and if\/ $e$ is even, then each level of\/ $\Gamma_s(\mathbf{B})$ has $0$ or $2$ self-dual modules.
\begin{proof}
Let\/ $\ell\in\{1,\dots,em\}$ and suppose that the $\ell$-th level of\/ $\Gamma_s(\mathbf{B})$ consists of the modules $Y_i$, for $i\in\Z$. Then $Y_{i+1}=\Omega^{-2}(Y_i)$ and\/ $Y_i\cong Y_j$ if and only if\/ $i\equiv j \pmod{e}$.

If\/ $e$ is odd, then level $\ell$ contains at least one self-dual module. If\/ $e$ is even, we may assume that level $\ell$ contains a self-dual module. In both cases we choose notation so that\/ $Y_0$ is self-dual. Then $Y_i=\Omega^{-2i}(Y_0)$ and isomorphism \eqref{E:dual-tau} gives $Y_i^*\cong Y_{e-i}$, for $i=0,\dots,e-1$. The congruence $i\equiv e-i \pmod{e}$ has one solution $i=0$ if\/ $e$ is odd, and the two solutions $i=0,e/2$, if\/ $e$ is even.
\end{proof}
\end{Lemma}

\begin{Corollary}\label{cor:SD_hooks_in_ARquiver}%
The block $\mathbf{B}$ has two isomorphism classes of self-dual hooks.

If\/ $H$ is a self-dual hook, then the other self-dual hook is $\Omega^e(H)$.

If\/ $e$ is even, $H$ and $\Omega^e(H)$ lie on the same end\/ $\Omega^2$-orbit of\/ $\Gamma_s(\mathbf{B})$. If\/ $e$ is odd, then $H$ lies on one end\/ $\Omega^2$-orbit and $\Omega^e(H)$ lies on the other end\/ $\Omega^2$-orbit.
\begin{proof}
To start with, we prove that\/ $\mathbf{B}$ contains at least one self-dual hook. First, it follows from the definition that the dual of a hook of\/ $\mathbf{B}$ is again a hook of\/ $\mathbf{B}$. Now, if\/ $\mathbf{B}$ contains a hook $H_0$ such that  $H_0^\ast\ncong H_0$, then $e>1$ and there is $1<i<e$ such that $H_0^\ast\cong \Omega^{-2i}(H_0)$.
Indeed, as duality acts as reflection in the real stem, $H_0$ and its dual must lie on the same end\/ $\Omega^2$-orbit of\/ $\Gamma_s(\mathbf{B})$. Then, the hook $\Omega^{-i}(H_0)$ is self-dual since
\[
\Omega^{-i}(H_0)^\ast\cong \Omega^i(H_0^\ast)\cong \Omega^{i-2i}(H_0)=\Omega^i(H_0)
\]
by isomorphism \eqref{E:dual-omega}.

Next it is clear that if\/ $H$ is a self-dual hook, then so is $\Omega^e(H)$ as
\[
   \Omega^e(H)^\ast\cong\Omega^{-e}(H^\ast)\cong \Omega^{-e}(H)\cong \Omega^{-e+2e}(H)= \Omega^e(H) 
\]
again by isomorphism \eqref{E:dual-omega}. 

If\/ $e$ is odd, then clearly $H$ and $\Omega^e(H)$ lie in different end \/ $\Omega^2$-orbits. If\/ $e$ is even, then $H$ and $\Omega^e(H)$ lie on the same end \/ $\Omega^2$-orbit of\/ $\Gamma_s(\mathbf{B})$ and $\Omega^{-2j+1}(H)$ cannot be self-dual for any $0\leq j\leq e-1$ by isomorphism \eqref{E:dual-omega}.  Finally, Lemma~\ref{lem:SD_levels_ARquiver} shows that $H$ and $\Omega^e(H)$ account for all the self-dual hooks of\/ $\mathbf{B}$.
\end{proof}
\end{Corollary}
\newpage

\subsection{Socle and radical layers and duality}

In order to determine the self-dual indecomposable modules we use the following well-known formulae for the socle and radical layers of the dual of a $kG$-module $X$. (See e.g. \cite[Section 6.3]{We}.) For each non-negative integer $n$, we have
\begin{equation}\label{form:layers_dual}
        \soc^{n+1}(X^\ast)/\soc^{n}(X^\ast)
        \cong 
        \big(\rad^{n}(X)/\rad^{n+1}(X)\big)^{\ast}\,.  \tag{F1}      
\end{equation}
In particular, the socle and the head of\/ $X$ and\/ $X^{\ast}$ are related by 
\begin{equation}\label{form:socle_head}
       \soc(X^\ast)\cong \hd(X)^\ast
       \qquad\text{ and }\qquad 
       \hd(X^\ast)\cong \soc(X)^\ast\,.  \tag{F2} 
\end{equation}

\section{Uniserial blocks}\label{sec:uniserial}%

We retain the assumptions of 2.1. In particular $D_1$ is the unique subgroup of\/ $D$ of order $p$ and\/ $\mathbf{B}_1$ is the Brauer correspondent of\/ $\mathbf{B}$ in $N_G(D_1)$. Now $D$ is also a defect group of\/ $\mathbf{B}_1$ and the Brauer tree $\sigma(\mathbf{B}_1)$ is a star, with exceptional central vertex and\/ $e$ edges attached to this vertex. We adopt the following additional notation. The edges (and irreducible $\mathbf{B}_1$-modules) are labelled\/ $E_0,\dots,E_{e-1}$, in counterclockwise order. We extend notation by defining $E_z:=E_{\ov z}$, where $\ov z$ is the residue of\/ $z$ modulo $e$, for all $z\in\Z$.
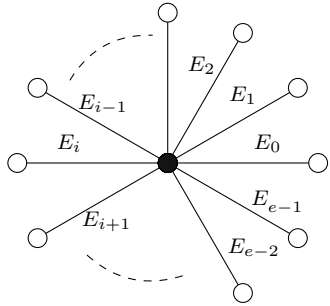
\[  
\begin{tikzpicture}[solidnode/.style={circle, fill=black!92, inner sep=0pt, minimum size=2.5mm},hollownode/.style={circle, fill=white!84, inner sep=0pt, minimum size=2.5mm}, auto,bend left]
        \draw (4,2)   node[hollownode,draw]{} -- node[above right] {\tiny $E_0$} (2,2) node[solidnode,draw]{};
        \draw (3.73,3)  node[hollownode,draw]{} --  (2,2) node[solidnode,draw]{};
        \draw (2,2)   node[solidnode,draw]{} --  node[right] {\tiny $\,\,E_{e-1}^{\phantom{X}}$} (3.73,1) node[hollownode,draw]{}  ;
        \draw (2,2) node[solidnode,draw]{} -- node[above right] {} (3,3.73) node[hollownode,draw]{} ;
        \draw (2,2) node[solidnode,draw]{} -- node[below right] {\tiny $\phantom{I}E_{e-2}$} (3,0.27) node[hollownode,draw]{} ;
        \draw (2,2) node[solidnode,draw]{} -- node[above] {} (2,4) node[hollownode,draw]{} ; 
        \draw (2,2) node[solidnode,draw]{} -- node[above left] {\tiny $E_i$} (0,2) node[hollownode,draw]{};
        \draw (2,2) node[solidnode,draw]{} -- node[above] {\tiny $\,\,E_{i-1}\,\,$} (0.27,3) node[hollownode,draw]{};
        \draw (2,2) node[solidnode,draw]{} -- node[below] {\tiny $\,\,E_{i+1}$} (0.27,1) node[hollownode,draw]{};
          \draw[dashed] (2.2,0.5) to (0.9,0.8);
          \draw[dashed] (0.7,3.1) to (1.8,3.7);
        \node at (3,2.9) {\tiny $E_1$};
        \node at (2.45,3.3) {\tiny $E_2$};
\end{tikzpicture}  
\] 

Each non-projective indecomposable module in $\mathbf{B}$ has a Green vertex $D_i$, for some $i=1,\dots,a$. As $\op{N}_G(D_i)\leq\op{N}_G(D_1)$, Green correspondence establishes a bijection between the non-projective indecomposable $\mathbf{B}$-modules and the non-projective indecomposable $\mathbf{B}_1$-modules. It also induces a graph isomorphism $\Gamma_s(\mathbf{B})\cong\Gamma_s(\mathbf{B}_1)$. Moreover, we have:

\begin{Lemma}\label{L:self-dual-green}
Green correspondence restricts to a type-preserving bijection between the self-dual indecomposable $\mathbf{B}$-modules and the self-dual indecomposable $\mathbf{B}_1$-modules.
\begin{proof}
Let\/ $M$ be a self-dual indecomposable $\mathbf{B}$-indecomposable module which has vertex $D_i$, where $D_i\ne1$. Then $M{\downarrow_{\op{N}_G(D_1)}}=N\oplus X$, where $N$ is the Green correspondent of\/ $M$ with respect to $D_i$ in $\op{N}_G(D_1)$. So $N$ belongs to $\mathbf{B}_1$ and has vertex $D_i$ and no indecomposable summand of\/ $X$ has vertex $D_i$. Taking duals, we get\/ $$M{\downarrow_{\op{N}_G(D_1)}}\cong M^*{\downarrow_{\op{N}_G(D_1)}}=N^*\oplus X^*\,.$$ As $N^*$ has vertex $D_i$, it follows that\/ $N^*\cong N$.

The types of\/ $M$ and\/ $N$ coincide, according to the main result of \cite{GW}.
\end{proof}
\end{Lemma}

As $\sigma(\mathbf{B}_1)$ is a star, $\mathbf{B}_1$ is a uniserial algebra. This means that each indecomposable $\mathbf{B}_1$-module is uniserial, that is, has a unique composition series. One consequence of this is that we may assume that the modules on the $\ell$-th level of the stable Auslander-Reiten quiver of\/ $\mathbf{B}_1$ have composition length $\ell$, for each  $\ell=1,\dots,em$. (See e.g. \cite[\S6.4]{B}.)

We use $[i,\ell]$ to denote  the unique uniserial $\mathbf{B}_{1}$-module with $\ell$ composition factors $E_i,E_{i-1}$, $E_{i-2},\dots,E_{i-\ell+1}$, listed from socle to head. Notice that\/ $[i,\ell]$ is a submodule of\/ $[j,k]$ if and only if\/ $i=j$ and\/ $\ell\leq k$, while $[i,\ell]$ is a factor module of\/ $[j,k]$ if and only if\/ $i-\ell\equiv j-k$ mod\/ $e$ and\/ $\ell\leq k$.

For each residue $i$ modulo $e$, define $i^*$ so that\/ $E_i^*\cong E_{i^*}$. It follows from \cite{Gr} that
$$
i^*=0^*-i,\quad\mbox{for all $i$.}
$$
Moreover, we may choose notation so that\/ $E_0^*\cong E_0$ or $E_0^*\cong E_{-1}$. In the first case $i^*=-i$ and in the second case $i^*=-i-1$.

\begin{Lemma}\label{lem:sdB_1general}
The block\/ $\mathbf{B}_1$ has $|D|-1$ isomorphism classes of non-projective self-dual indecomposable modules.

More precisely, $[i,\ell]$ is self-dual if and only if $\ell\equiv i-i^*+1\,\pmod{e}$.
\begin{proof}
As $[i,\ell]$ has head\/ $E_{i-\ell+1}$, (\ref{form:layers_dual}) implies that\/ $[i,\ell]^*$ has socle $E_{i-\ell+1}^*$. Moreover $[i,\ell]^*$ has the same composition length, namely $\ell$, as $[i,\ell]$. It follows that  $[i,\ell]^*\cong [(i-\ell+1)^*,\ell]$. This gives the congruence condition which determines if\/ $[i,\ell]$ is self-dual.

For each of the $e$ possible values of\/ $i$, there will be $m$ values of\/ $\ell$ satisfying the stated congruence condition. This gives the total count of\/ $em=|D|-1$ non-projective self-dual indecomposable $\mathbf{B}_1$-modules.
\end{proof}
\end{Lemma}

\begin{Corollary}\label{cor:numberselfduals}%
The block\/ $\mathbf{B}$ has exactly $|D|-1$ isomorphism classes of non-projective self-dual indecomposable modules. 
\begin{proof}
This follows from Lemmas \ref{L:self-dual-green} and \ref{lem:sdB_1general}.
\end{proof}
\end{Corollary}

From \cite{M}, $\mathbf{B}$ (and thus also $\mathbf{B}_1$) has a real exceptional character if and only if\/ $\mathbf{b}$ is real and\/ $e$ is even, or $\mathbf{b}$ is not real and\/ $e$ is odd. Also, all exceptional characters are real if any one of them is real. Finally, if\/ $e$ is even but\/ $\mathbf{b}$ is not real (so $\mathbf{B}_1$ has no real exceptional characters), and the real stem of\/ $\mathbf{B}_1$ contains two edges, then the Frobenius-Schur indicators of the two real (non-exceptional) irreducible characters in $\mathbf{B}_1$ are $+1$ and\/ $-1$, in some order.

We distinguish three cases, depending on the number of self-dual irreducible $\mathbf{B}_1$-modules:

\subsection{Case 1: $\mathbf{B}_1$ has no self-dual irreducible modules}

Then $e$ is even and the real stem of\/ $\sigma(\mathbf{B}_1)$ consists of the exceptional vertex. Set\/ $h:=e/2$ and assume that\/ $0^*=-1$. 

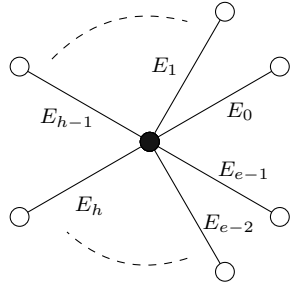
\[
\begin{tikzpicture}[solidnode/.style={circle, fill=black!92, inner sep=0pt, minimum size=2.5mm},hollownode/.style={circle, fill=white!84, inner sep=0pt, minimum size=2.5mm}, auto,bend left]
        \draw (3.73,3)  node[hollownode,draw]{} --  (2,2) node[solidnode,draw]{};
        \draw (2,2)   node[solidnode,draw]{} --  (3.73,1) node[hollownode,draw]{}  ;
        \draw (2,2) node[solidnode,draw]{} -- node[above right] {} (3,3.73) node[hollownode,draw]{};
        \draw (2,2) node[solidnode,draw]{} -- node[right] {} (3,0.27) node[hollownode,draw]{};
        \draw (2,2) node[solidnode,draw]{} -- (0.27,3) node[hollownode,draw]{};
        \draw (2,2) node[solidnode,draw]{} -- (0.27,1) node[hollownode,draw]{};
        \node at (3.2,2.4) {\tiny $E_0$};
        \node at (3.25,1.6) {\tiny $E_{e-1}$};
        \node at (2.2,3) {\tiny $E_1$};
        \node at (3.05,0.9) {\tiny $E_{e-2}$}; 
        \node at (0.9,2.3) {\tiny $E_{h-1}$};
        \node at (1.2,1.15) {\tiny $E_h$};
        \draw[dashed] (2.5,0.45) to (0.9,0.8);
        \draw[dashed] (0.7,3.1) to (2.5,3.6);
\end{tikzpicture}  
\]

Then $i^*\equiv-1-i \pmod{e}$, and in particular $h^*=h-1$. Also $[i,\ell]$ is self-dual if and only if $\ell\equiv 2i+2\pmod{e}$, by Lemma \ref{lem:sdB_1general}.

\begin{Proposition}\label{prop:sd(B_1)=0}
The non-projective self-dual indecomposable $\mathbf{B}_1$-modules are divided into two families:
{\begin{enumerate}[\rm(1)]
\item
The modules $[i,2i+2]$, for $i=0,1,\dots,mh-1$:
$$
\begin{tikzcd}[column sep=tiny]
\boxed{\substack{E_0^*\\E_0}} &&
\boxed{\substack{E_1^*\\E_0^*\\E_0\\E_1}} &&
\boxed{\substack{E_2^*\\E_1^*\\E_0^*\\E_0\\E_1\\E_2}} &&
\ldots
\end{tikzcd}
$$
These modules have the same type as the hook module $[(mh)^*,me]$.
\item
The modules $[h+i,2i+2]$, for $i=0,1,\dots,mh-1$:
$$
\begin{tikzcd}[column sep=tiny]
\boxed{\substack{E_h^*\\E_h}} &&
\boxed{\substack{E_{h+1}^*\\E_h^*\\E_h\\E_{h+1}}} &&
\boxed{\substack{E_{h+2}^*\\E_{h+1}^*\\E_h^*\\E_h\\E_{h+1}\\E_{h+2}}} &&
\ldots
\end{tikzcd}
$$
These modules have the same type as the hook module $[(mh+h)^*,me]$.
\end{enumerate}
There is one self-dual module from each family in each even numbered level of\/ $\Gamma_s(\mathbf{B}_1)$, and none in each odd numbered level.}
\end{Proposition}
\begin{proof}
We note that\/ $0^*\equiv-1$ and\/ $h^*\equiv h-1\pmod{e}$. There are no self-dual projective indecomposable modules in this case, because such a module would be the projective cover of a self-dual irreducible module. 

The projective cover of\/ $E_{0^*}$ has Jacobson radical $[0^*,me]$. This module is self-dual as $0\equiv me\equiv 2\times0^*+2\pmod{e}$\,. Let\/ $B_0$ be a non-degenerate $G$-invariant bilinear form on $[0^*,me]$. The submodule $[0^*,me-i]$ of\/ $[0^*,me]$ has $B_0$-dual $[0^*,i]$. So $B_0$ induces a non-degenerate $G$-invariant bilinear form (of the same type as $B_0$) on the factor module $[0^*,me-i]/[0^*,i]=[i^*,me-2i]$. When $i=mh-1$, we get the shortest module $[mh,2]$ in the family. These self-dual subquotients of\/ $[0^*,me]$ give all modules in one of the families.

Similarly the Jacobson radical $[h^*,me]$ of the projective cover of\/ $E_h^*$ is self-dual. Its subquotients give the modules $[(h+i)^*,me-2i]$ in the other family. When $i=mh-1$, we get the shortest module $[(m+1)h,2]$ in the family.

This accounts for all non-projective self-dual indecomposable modules in $\mathbf{B}_1$. The last statement is clear as the level of a module in $\Gamma_s(\mathbf{B}_1)$ is given by its composition length.
\end{proof}

\begin{Example}
Consider the group $C_{15}:C_8$, where a generator of\/ $C_8$ squares a generator of\/ $C_{15}$. This group has a real $5$-block\/ $\mathbf{B}$ of defect\/ $1$ which has no real irreducible characters. The Brauer tree is a star with $e=2=m$.
If\/ $S$ is an irreducible $\mathbf{B}$-module, then the irreducible $\mathbf{B}$-modules are $S$ and\/ $S^*$. There is a self-dual indecomposable $\mathbf{B}$-module with composition series $S/S^*$ and another with composition series $S^*/S$. One of these is of\/ $+$ type and the other is of\/ $-$ type.
$$
   \begin{tikzpicture}[solidnode/.style={circle, fill=black!92, inner sep=0pt, minimum size=2.5mm},hollownode/.style={circle, fill=white!84, inner sep=0pt, minimum size=2.5mm}, auto,bend left]
        \draw (0,0) node[solidnode,draw]{} -- node[right] {\tiny $S$} (0,1) node[hollownode,draw]{} ;
        \draw (0,0) node[solidnode,draw]{} -- node[right] {\tiny $S^\ast$} (0,-1) node[hollownode,draw]{} ;
    \end{tikzpicture}  
$$ 
\end{Example}

\begin{Example}
Consider the group $C_5:C_4$, where a generator of the $C_4$ inverts $C_5$. This group has a real $5$-block\/ $\mathbf{B}$ of defect\/ $1$ which has a pair of real exceptional characters but no real non-exceptional characters. The Brauer tree is a star with $e=2=m$. Let\/ $S$ be an irreducible $\mathbf{B}$-module. So the other irreducible $\mathbf{B}$-module is $S^*$. The block has $4$ self-dual indecomposable modules, all of\/ $-$ type. It should be noted that both exceptional characters have Frobenius-Schur indicator $-1$.
$$
   \begin{tikzpicture}[solidnode/.style={circle, fill=black!92, inner sep=0pt, minimum size=2.5mm},hollownode/.style={circle, fill=white!84, inner sep=0pt, minimum size=2.5mm}, auto,bend left]
        \draw (0,0) node[solidnode,draw]{} -- node[right] {\tiny $S$} (0,1) node[hollownode,draw]{} ;
        \draw (0,0) node[solidnode,draw]{} -- node[right] {\tiny $S^\ast$} (0,-1) node[hollownode,draw]{} ;
    \end{tikzpicture}  
$$ 
\end{Example}

\subsection{Case 2: $\mathbf{B}_1$ has two self-dual irreducible modules}

Then $e$ is even and again we set\/ $h:=e/2$. We may assume that\/ $E_0$ and\/ $E_h$ are the self-dual irreducible $\mathbf{B}_1$-modules. So the real stem of\/  $\sigma(\mathbf{B}_1)$ consists of the edges $E_0$ and\/ $E_h$ and their endpoints. Moreover $E_i^*=E_{-i}$, for $i=0,\dots,e-1$.

$$
   
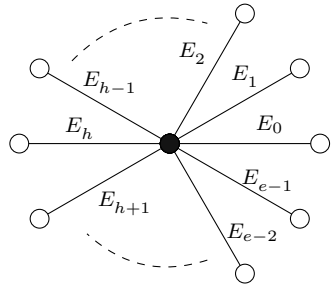
\begin{tikzpicture}[solidnode/.style={circle, fill=black!92, inner sep=0pt, minimum size=2.5mm},hollownode/.style={circle, fill=white!84, inner sep=0pt, minimum size=2.5mm}, auto,bend left]
        \draw (4,2)  node[hollownode,draw]{} -- node[above right] {\tiny $E_0$} (2,2) node[solidnode,draw]{};
        \draw (3.73,3)  node[hollownode,draw]{} --  (2,2) node[solidnode,draw]{};
        \draw (2,2)   node[solidnode,draw]{} --  (3.73,1) node[hollownode,draw]{} ;
        \draw (2,2) node[solidnode,draw]{} -- node[above right] {} (3,3.73) node[hollownode,draw]{} ;
        \draw (2,2) node[solidnode,draw]{} -- node[right] {} (3,0.27) node[hollownode,draw]{} ;
        \draw (2,2) node[solidnode,draw]{} -- (0,2) node[hollownode,draw]{} ;
        \draw (2,2) node[solidnode,draw]{} -- (0.27,3) node[hollownode,draw]{} ;
        \draw (2,2) node[solidnode,draw]{} -- (0.27,1) node[hollownode,draw]{} ;
        \node at (3,2.9) {\tiny $E^{}_1$};
        \node at (3.3,1.5) {\tiny $E^{}_{{e-1}}$};
        \node at (2.3,3.2) {\tiny $E^{}_2$};
        \node at (3.1,0.8) {\tiny $E_{e-2}$};
        \node at (1.2,2.8) {\tiny $E^{}_{h-1}$};
        \node at (1.4,1.2) {\tiny $E^{}_{h+1}$};
        \node at (0.8,2.2) {\tiny $E^{}_{h}$};
        \draw[dashed] (2.5,0.45) to (0.9,0.8);
        \draw[dashed] (0.7,3.1) to (2.5,3.6);
    \end{tikzpicture}  
$$ 

\begin{Proposition}\label{prop:sd(B_1)=2}
The non-projective self-dual indecomposable $\mathbf{B}_1$ modules are divided into two families:
\begin{enumerate}[\rm(1)]
\item
The modules $[i,2i+1]$, for $i=0,1,\dots,mh-1$:
$$
\begin{tikzcd}[column sep=tiny]
\boxed{\substack{E_0}} &&
\boxed{\substack{E_1^*\\E_0\\E_1}} &&
\boxed{\substack{E_2^*\\E_1^*\\E_0\\E_1\\E_2}} &&
\ldots
\end{tikzcd}
$$
These modules have the same type as that of\/ $E_0$ and of the projective cover of\/ $E_{mh}$.

\item
The modules $[h+i,2i+1]$, for $i=0,1,\dots,mh-1$:
$$
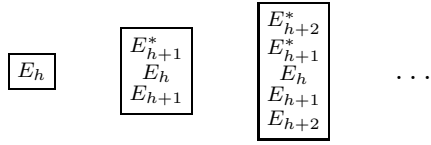
\begin{tikzcd}[column sep=tiny]
\boxed{\substack{E_h}} &&
\boxed{\substack{E_{h+1}^*\\E_h\\E_{h+1}}} &&
\boxed{\substack{E_{h+2}^*\\E_{h+1}^*\\E_h\\E_{h+1}\\E_{h+2}}} &&
\ldots
\end{tikzcd}
$$
These modules have the same type as that of\/ $E_h$ and of the projective cover of\/ $E_{(m+1)h}$.
\end{enumerate}
There is one self-dual module from each family in each odd numbered level of\/ $\Gamma_s(\mathbf{B}_1)$, and none in each even numbered level.
\begin{proof}
In this case $\mathbf{B}_1$ contains two non-isomorphic self-dual PIMs, namely the projective covers of\/ $E_0$ and\/ $E_h$. The self-dual subquotients of the projective cover of\/ $E_{mh}$ give all modules in the first family and the self-dual subquotients of the projective cover of\/ $E_{(m+1)h}$ give all modules in the second family. This determines the type of the modules in each family, as in the proof of Proposition \ref{prop:sd(B_1)=0}. We omit the details.

This accounts for all non-projective self-dual indecomposable modules in $\mathbf{B}_1$. The last statement is clear as the level of a module in $\Gamma_s(\mathbf{B}_1)$ is given by its composition length.
\end{proof}
\end{Proposition}

Note that the type of\/ $E_0$ coincides with the Frobenius-Schur indicator of the real non-exceptional character which lifts the Brauer character of\/ $E_0$. Likewise the type of\/ $E_h$ coincides with the Frobenius-Schur indicator of the other real non-exceptional character.

\begin{Corollary}\label{cor:all_same_type}%
All self-dual indecomposable\/ $\mathbf{B}_1$-modules have the same type, if\/ $m$ is odd.
\begin{proof}
The hypothesis implies that\/ $mh\equiv h\pmod{e}$. So all modules in the first family, including $E_0$, have the same type as the projective cover of\/ $E_h$. On the other hand, \cite[Proposition 2.2]{W} shows that\/ $E_h$ and its projective cover have the same type. So all self-dual indecomposable $\mathbf{B}_1$-modules have the same type as $E_0$.
\end{proof}
\end{Corollary}

\begin{Example}
Consider the group $C_{15}:C_4$, where a generator of the $C_4$ squares a generator of\/ $C_{15}$. This group has a real $5$-block\/ $\mathbf{B}$ of defect\/ $1$ which has two exceptional characters, one with Frobenius-Schur indicator $+1$ and the other with Frobenius-Schur indicator $-1$. The two exceptional characters form a complex conjugate pair. The Brauer tree is a star with $e=2=m$. Then $\mathbf{B}$ has two self-dual irreducible modules $S$ and\/ $T$, where $S$ has $+$ type and\/ $T$ has $-$ type.
\[
   \begin{tikzpicture}[solidnode/.style={circle, fill=black!92, inner sep=0pt, minimum size=2.5mm},hollownode/.style={circle, fill=white!84, inner sep=0pt, minimum size=2.5mm}, auto,bend left]
        \draw (0,0) node[solidnode,draw]{} -- node {\tiny $T$} (1.5,0) node[hollownode,draw]{} ;
        \draw (-1.5,0) node[hollownode,draw]{}  -- node {\tiny $S$} (0,0) node[solidnode,draw]{} ;
    \end{tikzpicture}  
\]

\end{Example}

\subsection{Case 3: $\mathbf{B}_1$ has one self-dual irreducible module}

Then $e$ is odd and we may assume that\/ $E_0$ is the unique self-dual irreducible $\mathbf{B}_1$-module. Also $m$ is even, as $me=|D|-1$ is even. The real stem of\/ $\sigma(\mathbf{B}_1)$ consists of\/ $E_0$ and its endpoints, and\/ $E_i^*=E_{-i}$, for $i=0,\dots,e-1$. We set\/ $h:=(e-1)/2$. 

  
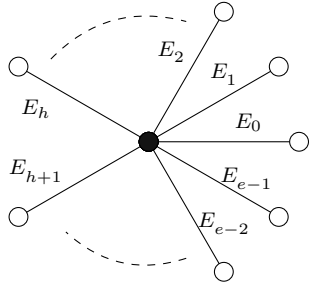
\[
   \begin{tikzpicture}[solidnode/.style={circle, fill=black!92, inner sep=0pt, minimum size=2.5mm},hollownode/.style={circle, fill=white!84, inner sep=0pt, minimum size=2.5mm}, auto,bend left]
        \draw (4,2)  node[hollownode,draw]{} -- node[above right] {\tiny $E_0$} (2,2) node[solidnode,draw]{};
        \draw (3.73,3)  node[hollownode,draw]{} --  (2,2) node[solidnode,draw]{};
        \draw (2,2)   node[solidnode,draw]{} --  (3.73,1) node[hollownode,draw]{} ;
        \draw (2,2) node[solidnode,draw]{} -- node[above right] {} (3,3.73) node[hollownode,draw]{} ;
        \draw (2,2) node[solidnode,draw]{} -- node[right] {} (3,0.27) node[hollownode,draw]{} ;
        \draw (2,2) node[solidnode,draw]{} -- (0.27,3) node[hollownode,draw]{} ;
        \draw (2,2) node[solidnode,draw]{} -- (0.27,1) node[hollownode,draw]{};
        \node at (3,2.9) {\tiny $E^{}_1$};
        \node at (3.3,1.5) {\tiny $E^{}_{e-1}$};
        \node at (2.3,3.2) {\tiny $E^{}_2$};
        \node at (3,0.9) {\tiny $E^{}_{e-2}$};
        \node at (0.5,2.44) {\tiny $E_h$};
        \node at (0.5,1.6) {\tiny $E^{}_{h+1}$};
        \draw[dashed] (2.5,0.45) to (0.9,0.8);
        \draw[dashed] (0.7,3.1) to (2.5,3.6);
    \end{tikzpicture}  
    \]

\begin{Proposition}\label{prop:sd(B_1)=1}%
The non-projective self-dual indecomposable $\mathbf{B}_1$-modules are divided into two families:
\begin{enumerate}[\rm(1)]
\item
The modules $[i,2i+1]$, for $i=0,1,\dots,me/2-1$:
$$
\begin{tikzcd}[column sep=tiny]
\boxed{\substack{E_0}} &&
\boxed{\substack{E_1^*\\E_0\\E_1}} &&
\boxed{\substack{E_2^*\\E_1^*\\E_0\\E_1\\E_2}} &&
\ldots
\end{tikzcd}
$$
These modules have the same type as the projective cover of\/ $E_0$.
\item
The modules $[h+i,2i+2]$, for $i=0,\dots,me/2-1$:
$$
\begin{tikzcd}[column sep=tiny]
\boxed{\substack{E_h^*\\E_h}} &&
\boxed{\substack{E_{h+1}^*\\E_h^*\\E_h\\E_{h+1}}} &&
\boxed{\substack{E_{h+2}^*\\E_{h+1}^*\\E_{h}^*\\E_{h}\\E_{h+1}\\E_{h+2}}} &&
\ldots
\end{tikzcd}
$$
These modules have the same type as the hook module $[h^*,me]\,\cong \Omega^e(E_0)$.
\end{enumerate}
There is one module in the first (respectively second) family on each odd (respectively even) numbered level of\/ $\Gamma_s(\mathbf{B}_1)$.
\begin{proof}
There is exactly one self-dual PIM in $\mathbf{B}_1$, namely the projective cover of\/ $E_0$. The self-dual proper subquotients of\/ $P(E_0)$ have the same type and give all modules in the first family. The smallest of these modules is $E_0$.

Similarly the Jacobson radical $[h^*,me]$ of the projective cover of\/ $E_h^*$ is self-dual, as its head is $E_h$ and its socle is $E_h^*$. The self-dual subquotients of\/ $[h^*,me]$ have the same type and give all modules in the second family. The smallest of these modules is $[h,2]$.

The isomorphism $[h^*,me]\,\cong \Omega^e(E_0)$ follows from Corollary \ref{cor:SD_hooks_in_ARquiver}.

This accounts for all non-projective self-dual indecomposable modules in $\mathbf{B}_1$. The last statement is clear as the level of a module in $\Gamma_s(\mathbf{B}_1)$ is given by its composition length.
\end{proof}
\end{Proposition}

\begin{Example}
The cyclic group $C_3$ has a single $3$-block of defect\/ $1$. The trivial character is non-exceptional and the two exceptional characters form a complex conjugate pair. The unique irreducible module (the trivial module) has $+$ type, and the unique 2-dimensional indecomposable module has $-$ type.
\[
   \begin{tikzpicture}[solidnode/.style={circle, fill=black!92, inner sep=0pt, minimum size=2.5mm},hollownode/.style={circle, fill=white!84, inner sep=0pt, minimum size=2.5mm}, auto,bend left]
        \draw (-1.5,0) node[hollownode,draw]{} -- node {\tiny $k$} (0,0) node[solidnode,draw]{} ;
    \end{tikzpicture}  
\]

\end{Example}

\begin{Example}
Let\/ $p=7$ and let\/ $\mathbf{B}$ be a self-dual $7$-block with a cyclic defect group $D\cong C_7$. Moreover, assume $e=3$, and thus $m=2$. The Brauer tree of the Brauer correspondent\/ $\mathbf{B}_1$ of\/ $\mathbf{B}$ in $N_G(D_1)$ is a star with three edges and exceptional vertex at its centre and there is precisely one self-dual simple $\mathbf{B}_1$-module, say $E_0$: 
  
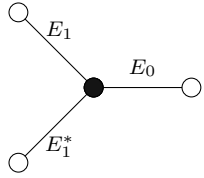
\[
   \begin{tikzpicture}[solidnode/.style={circle, fill=black!92, inner sep=0pt, minimum size=2.5mm},hollownode/.style={circle, fill=white!84, inner sep=0pt, minimum size=2.5mm}, auto,bend left]
        \draw (3,0)  node[solidnode,draw]{} -- node[above] {\tiny $E_0$} (4.3,0) node[hollownode,draw]{};
         \draw (2,1)  node[hollownode,draw]{} -- node[above] {\tiny $\,\,E_1$} (3,0) node[solidnode,draw]{};
         \draw (2,-1)  node[hollownode,draw]{} -- node[below] {\tiny $\,\,E_1^\ast$} (3,0) node[solidnode,draw]{};
    \end{tikzpicture} 
    \]
We write $E_1$ and  $E_1^\ast$ for the remaining pair of dual  irreducible modules, that is  $E_1^\ast=E_2$.  
The $|D|-1=6$ self-dual non-projective $\mathbf{B}_1$-modules can be described as follows:
\begin{enumerate}[\rm 1.]
    \item In family {\rm(1)}, we have 
    \[ 
    [1^*,5]={\boxed{\substack{E_1\\E_1^\ast\\E_0\\E_1\\E_1^\ast}}}\,\,,
    \quad
    [1,3]={\boxed{\substack{E_1^\ast\\E_0\\E_1}}}\,\,\text{ and }
    \quad
    [0,1]=E_0\,.
    \]
    \item
    In family {\rm(2)}, we have
    \[
    [1,6]={\boxed{\substack{E_1^\ast\\E_0\\E_1\\E_1^\ast\\E_0\\E_1}}}\,\,,
    \quad
    [0,4]={\boxed{\substack{E_0\\E_1\\E_1^\ast\\E_0}}}\,\,\text{ and }
    \quad
    [1^*,2]={\boxed{\substack{E_1\\E_1^\ast}}}\,.
    \]
\end{enumerate}
The stable Auslander-Reiten quiver $\Gamma_s(\mathbf{B}_1)$ is as follows. (Formally, we should draw arrows from from the third diagonal to the first diagonal, which close the tube of shape $(\mathbb{Z}/3\mathbb{Z})A_{6}$. We omit them, however, for clarity, as these are easy to work out.)

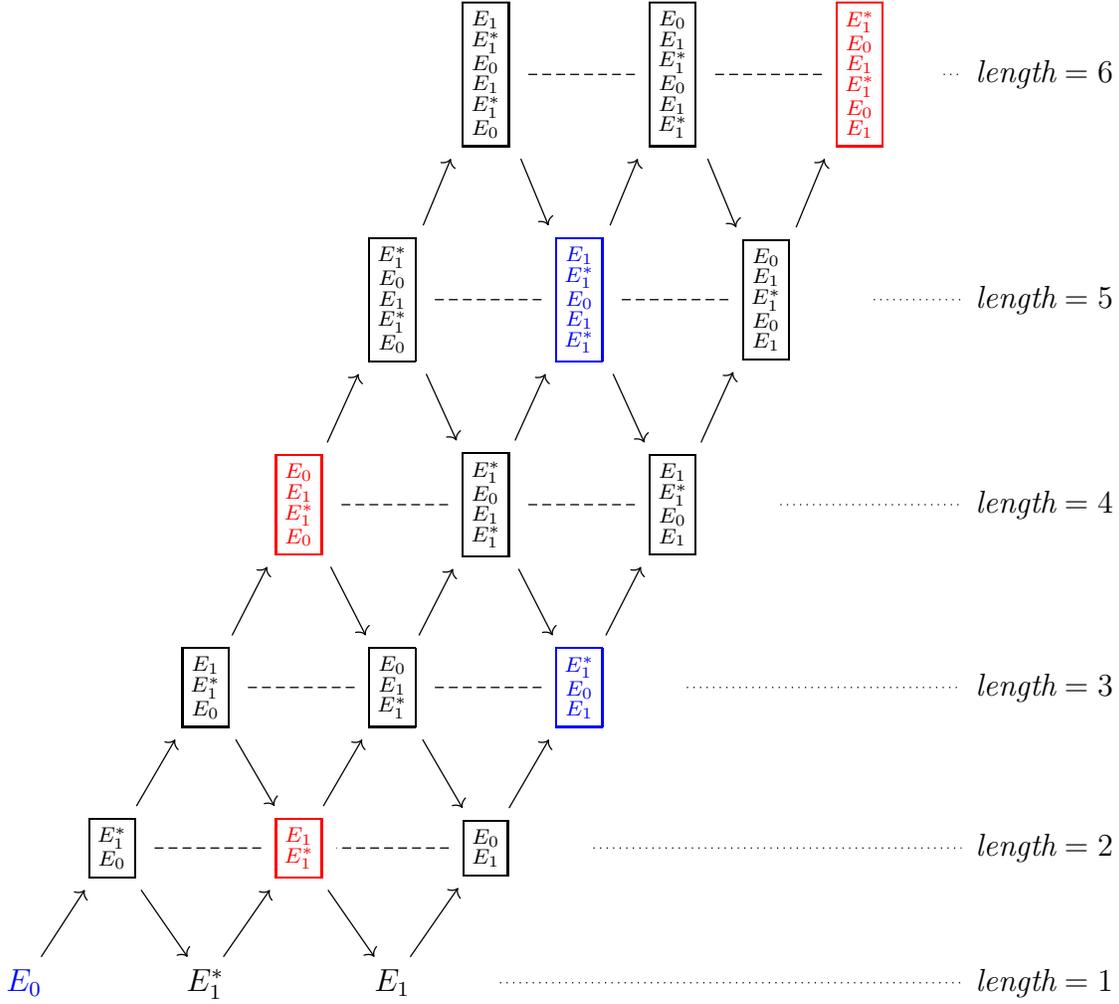
\tikzcdset{row sep/normal=0.9cm}
\[
\begin{tikzcd}[column sep=tiny]
	&&&&& {\boxed{\substack{E_1\\E_1^\ast\\E_0\\E_1\\E_1^\ast\\E_0}}} && {\boxed{\substack{E_0\\E_1\\E_1^\ast\\E_0\\E_1\\E_1^\ast}}} && {\color{red}{\boxed{\substack{E_1^\ast\\E_0\\E_1\\E_1^\ast\\E_0\\E_1}}}} & {} & {\text{length} =6} \\
	&&&& {\boxed{\substack{E_1^\ast\\E_0\\E_1\\E_1^\ast\\E_0}}} && {\color{blue}{\boxed{\substack{E_1\\E_1^\ast\\E_0\\E_1\\E_1^\ast}}}} && {\boxed{\substack{E_0\\E_1\\E_1^\ast\\E_0\\E_1}}} & {} && {\text{length} =5} \\
	&&& {\color{red}{\boxed{\substack{E_0\\E_1\\E_1^\ast\\E_0}}}} && {\boxed{\substack{E_1^\ast\\E_0\\E_1\\E_1^\ast}}} && {\boxed{\substack{E_1\\E_1^\ast\\E_0\\E_1}}} & {} & {} && {\text{length} =4} \\
	&& {\boxed{\substack{E_1\\E_1^\ast\\E_0}}} && {\boxed{\substack{E_0\\E_1\\E_1^\ast}}} && {\color{blue}{\boxed{\substack{E_1^\ast\\E_0\\E_1}}}} & {} &&&& {\text{length} =3} \\
	& {\boxed{\substack{E_1^\ast\\E_0}}} && {\color{red}{\boxed{\substack{E_1\\E_1^\ast}}}} && {\boxed{\substack{E_0\\E_1}}} & {} &&&&& {\text{length} =2} \\{\color{blue}
	{E_0}} && {E_1^\ast} && {E_1} & {} &&&&&& {\text{length} =1}
	\arrow[from=6-1, to=5-2]
	\arrow[from=5-2, to=4-3]
	\arrow[from=4-3, to=3-4]
	\arrow[from=3-4, to=2-5]
	\arrow[from=2-5, to=1-6]
	\arrow[from=6-3, to=5-4]
	\arrow[from=5-4, to=4-5]
	\arrow[from=4-5, to=3-6]
	\arrow[from=3-6, to=2-7]
	\arrow[from=2-7, to=1-8]
	\arrow[from=6-5, to=5-6]
	\arrow[from=5-6, to=4-7]
	\arrow[from=4-7, to=3-8]
	\arrow[from=3-8, to=2-9]
	\arrow[from=2-9, to=1-10]
	\arrow[from=5-2, to=6-3]
	\arrow[from=4-3, to=5-4]
	\arrow[from=3-4, to=4-5]
	\arrow[from=2-5, to=3-6]
	\arrow[from=5-4, to=6-5]
	\arrow[from=4-5, to=5-6]
	\arrow[from=3-6, to=4-7]
	\arrow[from=2-7, to=3-8]
	\arrow[from=1-6, to=2-7]
	\arrow[from=1-8, to=2-9]
	\arrow[dashed, no head, from=1-8, to=1-6]
	\arrow[dashed, no head, from=1-10, to=1-8]
	\arrow[dashed, no head, from=2-9, to=2-7]
	\arrow[dashed, no head, from=2-7, to=2-5]
	\arrow[dashed, no head, from=3-8, to=3-6]
	\arrow[dashed, no head, from=3-6, to=3-4]
	\arrow[dashed, no head, from=4-7, to=4-5]
	\arrow[dashed, no head, from=5-6, to=5-4]
	\arrow[dashed, no head, from=4-5, to=4-3]
	\arrow[dashed, no head, from=5-4, to=5-2]
	\arrow[dotted, no head, from=1-11, to=1-12
]
	\arrow[dotted, no head, from=2-10, to=2-12]
	\arrow[dotted, no head, from=3-9, to=3-12]
	\arrow[dotted, no head, from=4-8, to=4-12]
	\arrow[dotted, no head, from=5-7, to=5-12]
	\arrow[dotted, no head, from=6-6, to=6-12]
\end{tikzcd}
\]
 
\end{Example}

\subsection{Case when the defect group is normal}

In this section we assume that a defect group $D$ of\/ $\mathbf{B}$ is normal in $G$. We can then use an inner-product calculation to determine the types of the self-dual indecomposable $\mathbf{B}$-modules.

Recall that\/ $\mathbf{b}$ is a block of\/ $\op{C}_G(D)$ covered by $\mathbf{B}$. Set\/ $G_{\mathbf{b}}$ as its stabilizer in $G$. Then $e=[G_{\mathbf{b}}:\op{C}_G(D)]$ is the inertial index of\/ $\mathbf{B}$. Also the Brauer tree $\sigma(\mathbf{B})$ is a star with $e$ edges, indexing the irreducible $\mathbf{B}$-modules. The central vertex is exceptional, and it has multiplicity $m=\frac{|D|-1}{e}$.

As the quotient group $G_{\mathbf{b}}/\op{C}_G(D)$ is cyclic of\/ $p'$-order $e$, the unique irreducible $\mathbf{b}$-module has $e$ non-isomorphic extensions to $G_{\mathbf{b}}$. Moreover each extension induces irreducibly to $G$. These represent the $e$ isomorphism classes of irreducible $\mathbf{B}$-modules.

Let\/ $u\in D$ generate the unique order $p$ subgroup of\/ $D$. As $\langle u\rangle$ is normalized by $G$, for each $g\in G$ there is an integer $n(g)$, determined modulo $p$, such that\/ $gug^{-1}=u^{n(g)}$. It is clear that\/ $\lambda(g):=n(g)1_k$ defines a faithful linear $k$-character of\/ $G/\op{C}_G(D)$.

Let\/ $S$ be an irreducible $\mathbf{B}$-module and identify $\lambda$ with the corresponding $1$-dimensional $kG$-module. It follows from \cite[Section 5]{Gr} that\/ $S,\lambda S,\dots,\lambda^{e-1}S$ are representatives of the isomorphism classes of irreducible $\mathbf{B}$-modules. Moreover these index the edges of\/ $\sigma(\mathbf{B})$ in counter-clockwise order around\/ $0$. Thus $S\cong\lambda^e S$ and the $|D|$ composition factors of the (uniserial) projective cover of\/ $S$ are $S,\lambda S,\dots,\lambda^{me}S\cong S$, listed from top to bottom.

Let\/ $\chi$ and $\mu$ be irreducible characters of $G$ whose restrictions to the $p$-regular elements of\/ $G$ are the Brauer character of\/ $S$ and $\lambda$, respectively. Then $\chi,\mu\chi,\dots,\mu^{e-1}\chi$ are the non-exceptional irreducible characters in $\mathbf{B}$ and the restriction of\/ $\mu^i\chi$ to the $p$-regular elements of\/ $G$ is the Brauer character of\/ $\lambda^i S$, for $i=0,\dots,e-1$.

The statement of the following lemma is slightly more general than needed:

\begin{Lemma}\label{L:bilinear}
Let\/ $G$ be a finite group and let\/ $S$ be an irreducible $kG$-module, and let\/ $\lambda:G\rightarrow k^\times$ be a linear representation such that\/ $S^*\cong\lambda S$. Suppose that\/ $\mu$ and\/ $\chi$ are ordinary characters of\/ $G$ lifting the Brauer characters of\/ $\lambda$ and\/ $S$, respectively. Define
\begin{equation}\label{E:epsilon}
\epsilon_\mu(\chi):=\frac{1}{|G|}\sum_{g\in G}\mu(g)\chi(g^2).
\end{equation}
Then $\epsilon_\mu(\chi)=\pm1$, and there is a bilinear form $B:S\times S\rightarrow k$ such that
$$
\begin{aligned}
B(s_2,s_1)  &=&\epsilon_\mu(\chi)&B(s_1,s_2),\\
B(gs_1,gs_2)&=&\lambda(g)            &B(s_1,s_2),
\end{aligned}
\quad\mbox{for all $g\in G$ and\/ $s_1,s_2\in S$}.
$$
Moreover $B$ is uniquely determined by $S$ and\/ $\lambda$, up to a non-zero scalar.
\begin{proof}
Let\/ $\alpha:G\rightarrow\op{GL}_n(k)$ be a matrix representation for the $kG$-module $S$. Now the transpose inverse map $M\mapsto M^{-t}$, for $M\in\op{GL}_n(k)$, is an outer automorphism of $\op{GL}_n(k)$. Composing this with $\alpha$, we get\/ $\alpha^*:G\rightarrow\op{GL}_n(k)$, defined by $\alpha^*(g):=\alpha(g)^{-t}$, for all $g\in G$. This is a matrix representation corresponding to the dual $kG$-module $S^*$.

As $\lambda S\cong S^*$, there is $B\in\op{GL}_n(k)$, determined up to a non-zero scalar, such that
\begin{equation}\label{E:bilinear-form}
\alpha(g)^{-t}=\lambda(g)B\alpha(g)B^{-1},\quad\mbox{for all $g\in G$.}
\end{equation}
Applying the transpose inverse to both sides, we get
$$
\alpha(g)=\lambda(g)^{-1}\lambda(g)B^{-t}B\alpha(g)B^{-1}B^{t}=(B^{-t}B)\alpha(g)(B^{-t}B)^{-1}.
$$
As $\alpha$ is irreducible, Schur's Lemma implies that\/ $B^t=\zeta B$, for some $\zeta\in k$. Transposing both sides of this equation, we get\/ $B=\zeta^2 B$. So $\zeta=\pm1_k$.  Clearly $B$ is symmetric if\/ $\zeta=+1$, or alternating if\/ $\zeta=-1$.

Now \eqref{E:bilinear-form} can be rewritten as
$$
\alpha(g)^{-t}B\alpha(g)^{-1}=\lambda(g)B,\quad\mbox{for all $g\in G$.}
$$
So $kB$ is a $G$-invariant subspace of\/ $S^*\otimes_k S^*$ which is isomorphic to the 1-dimensional $kG$-module $\lambda$.

As $\op{char}(k)\ne2$, there is a $kG$-isomorphism
$$
S^*\otimes_k S^*\cong\op{Sym}(S)\oplus\op{Alt}(S),
$$
where $\op{Sym}(S)$ is the space of symmetric bilinear forms on $S$ and\/ $\op{Alt}(S)$ is the space of alternating bilinear forms on $S$. Now $\ov\chi^2=\ov\chi_s+\ov\chi_a$, where $\chi_s$ is the symmetric part of\/ $\chi^2$ and\/ $\chi_a$ is the alternating part of\/ $\chi^2$. Moreover, the restrictions of\/ $\ov\chi_s$ and\/ $\ov\chi_a$ to the $p$-regular elements of\/ $G$ give the Brauer characters of\/ $\op{Sym}(S)$ and\/ $\op{Alt}(S)$, respectively. Thus
$$
\zeta=\langle\mu,\ov\chi_s-\ov\chi_a\rangle=\frac{1}{|G|}\sum_{g\in G}\mu(g)\chi(g^2).
$$
\end{proof}
\end{Lemma}

Now we return to our block\/ $B$ with a cyclic normal defect group $D$, inertial index $e$ and irreducible module $S$. Also $S^*\cong\lambda S$, where $\lambda$ is the linear $k$-representation determined by the action of\/ $G$ on the elements of order $p$ in $D$. Since $\lambda S$ is the counter-clockwise successor of\/ $S$ in $\sigma(\mathbf{B})$ there is a non-split short exact sequence of\/ $\mathbf{B}$-modules
$$
0\rightarrow \lambda S\rightarrow M\rightarrow S\rightarrow 0,
$$
where $M$ is uniserial with composition length $2$, socle $\lambda S$ and head\/ $S$. Clearly $M$ is self-dual, as $\lambda S\cong S^*$.

\begin{Theorem}
With the notation above, the self-dual indecomposable $\mathbf{B}$-module $M=\begin{tikzcd}[column sep=tiny]\boxed{\substack{S\\\lambda S}}\end{tikzcd}$ has orthogonal type if\/ $\epsilon_\mu(\chi)=-1$, or symplectic type if\/ $\epsilon_\mu(\chi)=+1$.
\end{Theorem}

\begin{proof}
Let\/ $\alpha:G\rightarrow\op{GL}_n(k)$ be a matrix representation corresponding to $S$. Then corresponding to $M$, there is an indecomposable matrix representation $X:G\rightarrow\op{GL}_{2n}(k)$ which has block matrix form
$$
X(g)=\begin{bmatrix}\lambda(g)\alpha(g)&\beta(g)\\0&\alpha(g)\end{bmatrix},\quad\mbox{for all $g\in G$.}
$$
Here $\beta(g)$ is an $n\times n$-matrix with entries in $k$, for each $g\in G$.

Now $D$ is a normal $p$-subgroup of\/ $G$ and\/ $S$ and\/ $\lambda$ are irreducible $kG$-modules. Let\/ $d$ be a generator of the group $D$. Then $\alpha(d)=I$ and\/ $\lambda(d)=1_k$.  On the other hand\/ $D\not\le\op{ker}(M)$, as $M$ is indecomposable. So $X(d)=\begin{bmatrix}I&\beta(d)\\0&I\end{bmatrix}$, where $\beta(d)\ne0$.

Notice that\/ $X(d^m)=\begin{bmatrix}I&m\beta(d)\\0&I\end{bmatrix}$, for all $m\in\Z$. In particular $\langle d^p\rangle\leq\op{ker}(M)$. Thus $X(g)X(d)=X(d)^{\lambda(g)}X(g)$, for all $g\in G$ (recalling that\/ $\lambda(g)$ is defined modulo $p$ by the equation $gug^{-1}=u^{\lambda(g)}$). This equality implies in turn that
$$
\lambda(g)\alpha(g)\beta(d)=\beta(d)\lambda(g)\alpha(g),\quad\mbox{for all $g\in G$}.
$$
As $\alpha$ is irreducible and $\lambda(g)\in k^\times$, Schur's Lemma gives $\beta(d)=\sigma I$, for some $\sigma\in k^\times$. 

Now $M^*\cong M$ as $kG$-modules. So there is a matrix $E$ such that\/ $X(g)^{-t}E=EX(g)$, for all $g\in G$. We may also assume that\/ $E^t=\epsilon E$, where $\epsilon=\pm1$. Writing $E=\begin{bmatrix}A&B\\C&D\end{bmatrix}$, and set\/ $\gamma(g):=-\frac{1}{\lambda(g)}\alpha(g)^{-1}\beta(g)\alpha(g)^{-1}$, for all $g\in G$. Then
\begin{equation}\label{E:B}
\begin{bmatrix}\frac{1}{\lambda(g)}\alpha(g)^{-t}&0\\\gamma(g)^t&\alpha(g)^{-t}\end{bmatrix}\begin{bmatrix}A&B\\C&D\end{bmatrix}=
\begin{bmatrix}A&B\\C&D\end{bmatrix}\begin{bmatrix}\lambda(g)\alpha(g)&\beta(g)\\0&\alpha(g)\end{bmatrix}.
\end{equation}
First consider the case $g=d$. Then
$$
\begin{bmatrix}I&0\\-\sigma I&I\end{bmatrix}\begin{bmatrix}A&B\\C&D\end{bmatrix}=
\begin{bmatrix}A&B\\C&D\end{bmatrix}\begin{bmatrix}I&\sigma I\\0&I\end{bmatrix}.
$$
In particular $B=\sigma A+B$ and $-\sigma B+D=\sigma C+D$. The first equation gives $A=0$, and the second gives $C=-B$. Putting this back into \eqref{E:B}, we get\/ $\frac{1}{\lambda(g)}\alpha(g)^{-t}B=B\alpha(g)$, for all $g\in G$. This rearranges to \eqref{E:bilinear-form}. It then follows from Lemma \ref{L:bilinear} that\/ $B^t=\epsilon_\mu(\chi)B$, where $\chi$ and\/ $\mu$ are ordinary characters of $G$ which lift the Brauer characters of\/ $S$ and\/ $\lambda$, respectively, and $\epsilon_\mu(\chi)=\langle\lambda,\chi^{(2)}\rangle$ is defined by \eqref{E:epsilon}. Finally, as $C=-B$, we have
$$
\begin{bmatrix}0&\epsilon B\\-\epsilon B&\epsilon D\end{bmatrix}=\epsilon E=E^t=\begin{bmatrix}0&-\epsilon_\mu(\chi)B\\\epsilon_\mu(\chi)B&D^t\end{bmatrix}.
$$
So $\epsilon =-\epsilon_\mu(\chi)$ and $D^t=\epsilon D$. We conclude that\/ $M$ has orthogonal or symplectic type, as\/ $\epsilon_\mu(\chi)=-1$ or $+1$, respectively.
\end{proof}

\begin{Example}
Consider\/ $G=\op{SmallGroup}(120,7)$, in GAP notation. This is isomorphic to $C_{15}:C_8$, where a generator of the $C_8$ squares a generator of the $C_{15}$. Then $G$ has a 5-block\/ $\mathbf{B}=\{X_{11},X_{12},X_{15},X_{16}\}$ with cyclic defect group $C_5$ and\/ $e=2$. There are two irreducible $B$-modules, $S$ and\/ $S^*$, say. We assume that the Brauer character of\/ $S$ is the restriction of\/ $X_{11}$ to the $5$-regular elements of\/ $G$ and the Brauer character of\/ $S^*$ is the restriction of\/ $X_{12}$ to the $5$-regular elements of\/ $G$. The characters $X_{16}=\ov X_{15}$ are exceptional. We may take $\lambda=X_3$ as the linear character. Then we find that\/ $\langle X_3,X_{11}^{(2)}\rangle=+1$ and\/ $\langle X_3,X_{12}^{(2)}\rangle=-1$. So one of the self-dual indecomposable modules $S/S^*$ in $\mathbf{B}$ has symplectic type while the other $S^*/S$ has orthogonal type.
\end{Example}

We note that in a general example of this type with $e$ even, we have a linear character $\lambda$ and an irreducible character $\chi$ such that\/ $\ov\chi=\lambda\chi$. Also $\lambda^e\chi=\chi$. So we will have to compute $\langle\lambda,\chi^{(2)}\rangle$ and\/ $\langle\lambda,\lambda^e\chi^{(2)}\rangle$. These inner-products can be different, in spite of the fact that\/ $(\lambda^e\chi)^2=\chi^2$.

\begin{Example}
Consider $G=C_5:C_4$, where a generator of the $C_4$ inverts a generator of the $C_5$. Then $G$ has a 5-block with cyclic defect group $C_5$ and\/ $e=2$. Again $\mathbf{B}$ has no self-dual irreducible modules. But now its two exceptional characters are real, with Frobenius-Schur indicator $-1$. So immediately all self-dual indecomposable $\mathbf{B}$-modules have symplectic type.

We can also work out the type, using the technique of this note. Let\/ $S$ and\/ $S^*$ be the irreducible $\mathbf{B}$-modules, with corresponding lifted ordinary characters $\chi$ and\/ $\ov\chi$. Now $\lambda$ is a real linear character of\/ $G$. We find that\/ $\langle \lambda,\chi^{(2)}\rangle=\langle \lambda,\ov\chi^{(2)}\rangle=+1$. So both self-dual indecomposable modules $S/S^*$ and\/ $S^*/S$ have symplectic type.
\end{Example}

\section{Self-dual modules: the general case}\label{sec:generalcase}%

We now turn to the description of the self-dual indecomposable modules and their type for the general case of a real $p$-block\/ $\mathbf{B}$ which has a cyclic defect group $D$ of order $p^a\geq p\geq 3$, inertial index $e$ and exceptional multiplicity $m$. 

If\/ $e=1$, then the self-dual modules and their type are described as in Section \ref{sec:uniserial}. Therefore, we may assume $e\geq2$. Moreover, throughout this section, we adopt the following  conventions. We denote by $\mathbf{B}_1$ the Brauer correspondent of\/ $\mathbf{B}$, by $f$ the Green correspondence from $\mathbf{B}$ to $\mathbf{B}_1$ and by $g$ the Green correspondence from $\mathbf{B}_1$ to $\mathbf{B}$. The irreducible\/ $\mathbf{B}$-modules are labelled\/ $E_0,\dots,E_{e-1}$.
In this section, the irreducible $\mathbf{B}_1$-modules are labelled\/ $E_0(\mathbf{B}_1),\dots,E_{e-1}(\mathbf{B}_1)$ and correspond to the edges of $\sigma(\mathbf{B}_1)$ in counterclockwise order around the central exceptional vertex as in Section~\ref{sec:uniserial}.

\begin{Notation}\label{nota:bkappa}
There is no general pattern for the shape of the Brauer tree $\sigma(\mathbf{B})$. However, we can choose our notation so that the edges split into three families as follows:
\begin{enumerate}[\rm (1)]
    \item $b$ edges $\{E_0,\ldots,E_{b-1}\}$ labelling the real stem with $0\leq b\leq e$ (clearly $b$ is such that\/ $e-b$ is even);
    \item $(e-b)/2$ edges $\{E_{b},\ldots,E_{b-1+(e-b)/2}\}$ lying above the real stem in the planar embedding of\/ $\sigma(\mathbf{B})$\,;
    \item $(e-b)/2$ edges $\{E_{b}^\ast,\ldots,E_{b-1+(e-b)/2}^\ast\}$ lying below the real stem in the  planar embedding of\/ $\sigma(\mathbf{B})$.
\end{enumerate}
Moreover, provided\/ $m\geq 2$, we let\/ $\kappa$ denote the number edges of\/ $\sigma(\mathbf{B})$ which are not on the real stem  and whose shortest path to the real stem connects them to a node which is distinct from the exceptional node. Then $e-b-\kappa$ is the number of  edges of\/ $\sigma(\mathbf{B})$ which are not on the real stem  and whose shortest path to the real stem connects them to the exceptional node. 
\end{Notation}

Now, recall that there is a bijection between the irreducible modules and the PIMs of\/ $\mathbf{B}$ given by the taking of projective covers. It follows from this and formula (\ref{form:layers_dual}) that a PIM of\/ $\mathbf{B}$ is self-dual if and only if it is the projective cover of a self-dual irreducible module. This means, that the total number of isomorphism classes of self-dual indecomposable $\mathbf{B}$-modules somehow depends on the shape of the Brauer tree\/ $\mathbf{B}$, and more precisely on the length of the real stem. However, the number of isomorphism classes of non-projective  self-dual indecomposable $\mathbf{B}$-modules depends only on $|D|$, namely by Corollary \ref{cor:numberselfduals} there are exactly $|D|-1$ of them.

With this information, it is easy to produce a list of self-dual $\mathbf{B}$-modules and prove that they account for all the self-dual $\mathbf{B}$-modules using a counting argument. The precise classification is given by the following theorem and gives us Theorem \ref{thm:intro} of the introduction.

\begin{Theorem}\label{thm:description_selfduals}
Let\/ $p$ be an odd prime number. Let\/ $\mathbf{B}$ be a real $p$-block which has a non-trivial cyclic defect group $D$, inertial index $e\geq 2$, exceptional multiplicity $m$, and Brauer tree $\sigma(\mathbf{B})$. Then, up to isomorphism, the self-dual indecomposable $\mathbf{B}$-modules are as listed below. 
\begin{enumerate}[\rm(a)]
\item 
For $m\geq 1$, the $b$ irreducible $\mathbf{B}$-modules labelling the edges of the real stem of\/ $\sigma(\mathbf{B})$ are self-dual, where possibly $0\leq b\leq e$.
\item 
For $m\geq 1$, the $b$ projective indecomposable $\mathbf{B}$-modules of\/ $\mathbf{B}$ which are the projective covers of the $b$ irreducible modules labelling the real stem f\/ $\sigma(\mathbf{B})$ are self-dual.
\item  If\/ $m=1$, then there are precisely $e-b$ isomorphism classes of non-projective non-irreducible self-dual indecomposable  $\mathbf{B}$-modules, each parametrised by a path of shape
    \[
             \xymatrix@R=0.0000pt@C=30pt{
                    {\Circle}\ar@<0.0ex>[r]^{S_{1}}&
                    {\Circle}\ar@{.}[r]&{\Circle}\ar@<0.0ex>[r]^{S_{c}}&
                    {\Circle}\ar@<0.0ex>[r]^{S_{c}^\ast}&
                    {\Circle}\ar@{.}[r]&
                    {\Circle}\ar@<0.0ex>[r]^{S_{1}^\ast}&
                    {\Circle}
                    }
    \]    
    with direction $(1,-1)$, multiplicity $\mu=0$, and where, moreover, either $S_1,\ldots, S_c\in \{E_{b},\ldots,E_{{b-1}+(e-b)/2}\}$ or $S_1,\ldots, S_c\in \{E_{b}^\ast,\ldots,E_{{b-1}+(e-b)/2}^\ast\}$.
    
   \item  If\/ $m\geq 2$, then the pairwise distinct isomorphism classes of non-projective non-irreducible indecomposable $\mathbf{B}$-modules  parametrized by the following paths of even length are self-dual: 
    \begin{enumerate}[\rm(i)]
        \item the  $\kappa$ paths of shape
        \[
             \xymatrix@R=0.0000pt@C=30pt{
                    {\Circle}\ar@<0.0ex>[r]^{S_{1}}&
                    {\Circle}\ar@{.}[r]&{\Circle}\ar@<0.0ex>[r]^{S_{c}}&
                    {\Circle}\ar@<0.0ex>[r]^{S_{c}^\ast}&
                    {\Circle}\ar@{.}[r]&
                    {\Circle}\ar@<0.0ex>[r]^{S_{1}^\ast}&
                    {\Circle}
                    }
    \] 
    with direction $(1,-1)$, multiplicity $\mu=0$, and where, moreover, either $S_1,\ldots, S_c\in \{E_{b},\ldots,E_{b-1+(e-b)/2}\}$ or $S_1,\ldots, S_c\in \{E_{b}^\ast,\ldots,E_{b-1+(e-b)/2}^\ast\}$\,;
        \item  the  $e-b-\kappa$ paths of shape        
            \[
             \xymatrix@R=0.0000pt@C=30pt{
                    {\Circle}\ar@<0.0ex>[r]^{S_{1}}&
                    {\Circle}\ar@{.}[r]&{\Circle}\ar@<0.0ex>[r]^{S_{c}}&
                    {\CIRCLE}\ar@<0.0ex>[r]^{S_{c}^\ast}&
                    {\Circle}\ar@{.}[r]&
                    {\Circle}\ar@<0.0ex>[r]^{S_{1}^\ast}&
                    {\Circle}
                    }
            \] 
         with direction $(1,-1)$, multiplicity $1\leq \mu\leq m$, and where, moreover, either $S_1,\ldots, S_c\in \{E_{b},\ldots,E_{b-1+(e-b)/2}\}$ or $S_1,\ldots, S_c\in \{E_{b}^\ast,\ldots,E_{b-1+(e-b)/2}^\ast\}$\,;
        \item   the  $\kappa$ paths of shape      
        \[ 
        \xymatrix@R=0.0000pt@C=29pt{
 	    & &&&& &\\
	           {\Circle}\ar@<0.0ex>[r]^{S_{1}}&{\Circle}\ar@{.}[r]&{\Circle}\ar@<0.0ex>[r]^{S_{c-1}}&{\Circle}\ar[ddr]^{S_{c}} & & &  \\
		&&&&&& & \\
		&&&&{\Circle}\ar[ddl]^{\:\:S_{c}^\ast}       \ar@<0.3ex>[r]^{S_{c+1}}&                         {\Circle}\ar@<0.3ex>[l]^{S_{c+1}}\ar@{.}[r]&                    {\Circle}\ar@<0.3ex>[r]^{S_{c+l}}&    {\CIRCLE}\ar@<0.3ex>[l]^{S_{c+l}}\\
		&&&&& &\\
		{\Circle}&{\Circle}\ar@<0.0ex>[l]^{S_{1}^\ast}\ar@{.}[r]&{\Circle}& {\Circle} \ar@<0.0ex>[l]^{S_{c-1}^\ast}& & & \\
		&&&&& &
	   }
         \] 
        with direction $(1,-1)$, multiplicity $2\leq \mu\leq m$, $l\geq 1$, and where, moreover, either $S_1,\ldots, S_c\in \{E_{b},\ldots,E_{b-1+(e-b)/2}\}$ or $S_1,\ldots, S_c\in \{E_{b}^\ast,\ldots,E_{b-1+(e-b)/2}^\ast\}$, and\/ $S_{c+1},\ldots,S_{c+l}\in \{E_0,\ldots, E_{b-1}\}$\,;
        \item   the $b$ paths of shape   
        \[ \xymatrix@R=0.0000pt@C=30pt{
                   {\Circle}\ar@<0.3ex>[r]^{S_{1}} &
                   {\Circle}\ar@<0.3ex>[l]^{S_1}\ar@{.}[r] &
                    {\Circle}\ar@<0.3ex>[r]^{S_{c}}&{\CIRCLE}\ar@<0.3ex>[l]^{S_c}
                   }
         \]
    with direction $(1,-1)$, multiplicity $2\leq\mu\leq m$, and\/ $S_1,\ldots,S_c\in \{E_0,\ldots, E_{b-1}\}$. 
    \end{enumerate}
\end{enumerate}
\end{Theorem}

\begin{proof}
Assertions (a) and (b) are straightforward. 
We also notice that all the paths given in (c) and (d) give rise to pairwise non-isomorphic self-dual modules by the description of the composition factors of the associated modules in Notation~\ref{nota:paths_dir_mult} and Formula~(\ref{form:layers_dual}) for the socle layers of the dual module. Now, by Corollary~\ref{cor:numberselfduals}, there are precisely $|D|-1=e$ isomorphism classes of non-projective indecomposable self-dual $\mathbf{B}$-modules. Therefore, it only remains to verify that the  numbers of modules described in (a) and (c) if\/ $m=1$, respectively in (a) and (d) if\/ $m\geq 2$, add up to $|D|-1$. 
\par
First, assume that\/ $m=1$. If\/ $b=e$, then there is nothing to do. Indeed, by (a) we have already $|D|-1=e$ pairwise non-isomorphic non-projective indecomposable $\mathbf{B}$-modules, namely the irreducible modules $E_0,\ldots,E_{e-1}$. If\/ $1\leq b < e$, then there exists $e-b$ paths as given in (c), starting at each edge not on the real stem. Adding up these modules with the $e$ irreducible modules from (a), we have already $(e-b)+b=e=|D|-1$ pairwise non-isomorphic self-dual indecomposable modules, proving Assertion (c).
\par
Next, assume that\/ $m\geq 2$. There are $\kappa$ isomorphism classes of modules of type (i), $(e-b-\kappa)m$ of type (ii), $\kappa(m-1)$ of type (iii), $b(m-1)$ of type $(iv)$. Thus, altogether (a) and (d) provide us with 
        \[  
        b+\kappa+(e-b-\kappa)m+\kappa(m-1)+b(m-1)=em=|D|-1
        \]
isomorphism classes of self-dual indecomposable $\mathbf{B}$-modules. Assertion (d) and the claim of the theorem follow.         
\end{proof}

Next we explain how to describe the type of the self-dual modules listed in {Theorem~\ref{thm:description_selfduals}}. It is clear that the type of a self-dual indecomposable $\mathbf{B}$-module is not an invariant of the Morita equivalence class of\/ $\mathbf{B}$. In other words, the data encoded in the Brauer tree is not sufficient to determine the type. However, as in the uniserial case, we can always compare self-dual indecomposables with self-dual hooks of\/ $\mathbf{B}$. In order to achieve this, we use concepts and results from \cite{BC} and \cite{HLa}. First it is necessary to define the side of the tube $\Gamma_s(\mathbf{B})=(\mathbb{Z}/e\mathbb{Z})A_{|D|-1}$ which we consider as \emph{level $1$}. This is unfortunately technical. We explain, however, how it works in the next remark in details. 

\begin{Rem}\label{rem:using_BC}
\begin{enumerate}[\rm(1)]
\item
The hooks of\/ $\mathbf{B}$ are all liftable modules and the ordinary characters afforded by their lifts are easy to describe. More precisely, if\/ $E$ is an irreducible $\mathbf{B}$-module labelling the edge 
$$
    \begin{tikzcd}
    \cdots\,\underset{\chi_a}{{\Circle}} \arrow[r, dash,"E"] & \underset{\chi_b}{{\Circle}} \,\cdots
    \end{tikzcd}
$$
of\/ $\sigma(\mathbf{B})$, then its projective cover $P(E)$ is
    \vspace{-1mm}
\[
    P(E)=\boxed{\begin{smallmatrix} E \\ Q_a\oplus\, Q_b \\ E\end{smallmatrix}}\,.
\]
with $\rad(P(E))/\soc(P(E))=Q_a\oplus Q_b$ for two uniserial (possibly zero) $\mathbf{B}$-modules {$Q_a$ and~$Q_b$}. The hooks of\/ $\mathbf{B}$ are defined to be the uniserial quotients
$$
H_a:=\boxed{\begin{smallmatrix}E\\Q_a\end{smallmatrix}}\qquad\text{ and }\qquad H_b:=\boxed{\begin{smallmatrix}E\\Q_b\end{smallmatrix}}\,.
$$
The projective indecomposable character afforded by $P(E)$ is $\Phi_{S_j}=\chi_{a}+\chi_{b}$, and as $e>1$, counting constituents, it easily follows from the decomposition matrix of\/ $\mathbf{B}$ that any lift of\/ $H_a$ affords the character $\chi_a$ and any lift of\/ $H_b$ affords the character~$\chi_b$. We refer for example to \cite{Gr} and \cite{HN} for proofs of these \smallskip results.
\item Let\/ $u$ be a generator of the subgroup $D_1$ of\/ $D$ of order $p$. Since projective characters vanish at\/ $p$-elements, we have $\chi_a(u)=-\chi_b(u)$. Moreover, these character values are integers. Hence, following \cite[\S4.2.]{HLa}, for $c\in\{a,b\}$ we say that the hook\/ $H_c$ is a \emph{positive hook} if\/ $\chi_c(u)>0$, respectively  a \emph{negative hook} if  $\chi_c(u)<0$.
\par
By \cite[Corollary 4.3]{HLa} a hook of\/ $\mathbf{B}$ is positive if its Green correspondent in $\mathbf{B}_1$ is irreducible, and it is negative if its Green correspondent in $\mathbf{B}_1$ has composition length $|D|-1$. In other words, one end\/ $\Omega^2$-orbit of\/ $\Gamma_s(\mathbf{B})$ consists of the positive hooks and the other end\/ $\Omega^2$-orbit of\/ $\Gamma_s(\mathbf{B})$ consists of the negative \smallskip hooks.
   
\item
For $w\in\Z$ write $\overline{w}:=w+e\Z$ for the residue class of $w$ in $\Z/e\Z$. We say that an indecomposable\/  $\mathbf{B}$-module\/ $Z$ is located at position\/  $(\,\overline{v},\ell)$ in\/ $\Gamma_S(\mathbf{B})\cong (\Z/e\Z)A_{|D|-1}$ if the Green correspondent\/ $f(Z)$ of\/ $Z$ is the unique uniserial\/ $\mathbf{B}_1$-module with  composition length $\ell$ and socle $E_w(\mathbf{B}_1)$ with $0\leq w\leq e-1$ and $\overline{w}=\overline{v}$ \smallskip in $\Z/e\Z$. 
   
\item 
Each non-projective indecomposable  $\mathbf{B}$-module $M$ defines a uniquely determined positive hook\/ $H_M^+$ of\/ $\mathbf{B}$ as follows. Assume we consider the positive hooks of\/ $\mathbf{B}$ constitute level $1$ of\/ $\Gamma_s(\mathbf{B})\cong (\mathbb{Z}/e\mathbb{Z})A_{|D|-1}$ and\/ $M$ is located at position $(\,\overline{v}, \ell)$. Then $H_M^+$ is defined to be the module at position $(\,\overline{v},1)$. So, $H_M^+$ lies at the end of a shortest path in $\Gamma_s(\mathbf{B})$ from $M$ to the rim consisting of the positive hooks. Still following \cite[\S4.2.]{HLa}  we define the \emph{positive distance} of\/ $M$ to the rim, written
$$
d^{+}(M,H_M^+)\,,
$$
to be the length of the shortest path in $\Gamma_s(\mathbf{B})$ between $M$ and\/ $H_M^+$. Furthermore, we let\/ $d^{-}(M,H_M^+)$ to be the length of the shortest path in $\Gamma_s(\mathbf{B})$ between $M$ and\/ $\Omega^1(H_M^+)$. We have \smallskip $d^{+}(M,H_M^+)+d^{-}(M,H_M^+)=|D|-2$.
    
\item The main results of \cite{BC}  provide us with closed formulae to compute the distances $d^+(M,H_M^+)$ and\/ $d^{-}(M,H_M^+)$ from $M$ to the rim of  $\Gamma_s(\mathbf{B})$. These formulae are based on the description of the modules by their path, direction and multiplicity (if not irreducible). They are very technical and involve a large number of further  parameters. We do not repeat them here, but refer the reader directly to \cite[Theorem 3.3 and Theorem 3.5]{BC}.
\end{enumerate}
\end{Rem}

\begin{Theorem}\label{thm:type_selfduals_general_case}%
Let\/ $p$ be an odd prime number. Let\/ $\mathbf{B}$ be a real $p$-block which has a non-trivial cyclic defect group $D$, inertial index $e\geq 2$, exceptional multiplicity $m$, and Brauer tree $\sigma(\mathbf{B})$. Let\/ $\mathbf{B}_1$ denote the Brauer correspondent of\/ $\mathbf{B}$ in $N_G(D_1)$.  

Let\/ $M$ be a non-projective self-dual indecomposable $\mathbf{B}$-module and assume the hook\/ $H_M^+$ as well as $d^+(M,H_M^+)$  have been determined using \cite[Theorem 3.3 and Theorem 3.5]{BC}. Then the following assertions hold.
\begin{enumerate}[\rm(a)]
\item
Assume $e$ is even and  $\mathbf{B}_1$ has two pairwise non-isomorphic  self-dual irreducible modules. Then, there exists an integer $0\leq i<(|D|-1)/2$ such that\/ $d^+(M,H_M^+)=:2i$  and the type of\/ $M$ is equal to the type of the self-dual hook\/ $\Omega^{-2i}(H_M^+)$. If\/ $m$ is odd, then all self-dual indecomposable\/ $\mathbf{B}$-modules have the same type. 

\item
Assume $e$ is even and\/ $\mathbf{B}_1$ has no self-dual irreducible modules. Then there exists an integer  $0\leq j<(|D|-1)/2$ such that\/ $d^-(M,H_M^+)=:2j$ and the type of\/ $M$ is equal to the type of the self-dual hook\/ $\Omega^{2j+1}(H^+_M)$.

\item
Assume $e$ is odd. If\/ $E_0(\mathbf{B}_1)$ denotes the unique self-dual irreducible $\mathbf{B}_1$-module, then the following holds:
\begin{enumerate}[\rm(i)]
\item if\/ $d^+(M,H_M^+)$ is even, then the type of\/ $M$ is equal to the type of the self-dual hook\/ $g(E_0(\mathbf{B}_1))$\,; and

\item if\/ $d^+(M,H_M^+)$ is odd, then the type of\/ $M$ is equal to the type of the self-dual hook\/ $\Omega^e(g(E_0(\mathbf{B}_1)))$.
\end{enumerate}
\end{enumerate}
\end{Theorem}

\begin{proof}
By Lemma \ref{L:self-dual-green} Green correspondence between $\mathbf{B}$ and\/ $\mathbf{B}_1$ preserves the type of the self-dual indecomposable modules. Thus, $M$ and\/ $f(M)$ have the same type, and so do $H_M^+$ and\/ $f(H_M^+)$. Moreover, in $\mathbf{B}_1$ we have
$$
f(H_M^+)\cong H^+_{f(M)}\cong\soc(f(M))
$$
and\/ $d^+(M,H_M^+)=d^+(f(M),H^+_{f(M)})$ is the composition length of\/ $f(M)$ minus one. In addition, the Green correspondence commutes with the Heller operator, so it suffices to prove the theorem for the Green correspondents in $\mathbf{B}_1$ of the modules involved. 
\begin{enumerate}[\rm(a)]
\item
With the notation of Proposition~\ref{prop:sd(B_1)=2},  either $f(M)=[i,2i+1]$ with $0\leq i\leq mh-1$, or $f(M)=[h+i,2i+1]$ with $0\leq i \leq mh-1$. In both cases, the composition length is $2i+1$. Moreover, in the former case the type of\/ $f(M)$ is the type of\/
$$
E_0(\mathbf{B}_1)\cong\Omega^{-2i}(\soc(f(M)))=\Omega^{-2i}(f(H_M^+))\,,
$$ and in the latter case the type of\/ $f(M)$ is the type of\/
$$
E_h(\mathbf{B}_1)=\Omega^{-2i}(\soc(f(M)))\cong\Omega^{-2i}(f(H_M^+))\,.
$$
The second claim follows directly from Corollary \ref{cor:all_same_type}.  
\item
With the notation of Proposition~\ref{prop:sd(B_1)=0},  either $f(M)=[i,2i+2]$ with $0\leq i\leq mh-1$, or $f(M)=[h+i,2i+2]$ with $0\leq i \leq mh-1$. In both cases, the composition length is $2i+2$. So $d^-(f(M),H_M^+)=2j$ for $j:=(|D|-3)/2-i$ and in the former case the type of\/ $f(M)$ is the type of the self-dual hook 
$$
[((mh)^\ast,me)]\cong \Omega^{2j+1}(\soc(f(M)))\cong\Omega^{2j+1}(f(H_M^+))\,,
$$
and in the latter case the type of\/ $f(M)$ is the type of the self-dual hook
$$
[((mh+h)^\ast,me)]\cong \Omega^{2j}(\Omega(\soc(f(M))))\cong\Omega^{2j+1}(f(H_M^+))\,.
$$
 
\item
Assertions (i) and (ii) are immediate from the above and  Proposition \ref{prop:sd(B_1)=1}.
\end{enumerate}
\end{proof}

\begin{Example}[A real block with $e=12$ and\/ $m=9$]
Consider the principal $109$-block\/ $\mathbf{B}_0$ of the Ree group $^2\!\text{F}_4(q^2)$ with $q^2=2^{2u+1}$ for $u:=1$. Its planar embedded Brauer tree was determined in~\cite[Theorem~4.7]{H}. A defect group $D$ is cyclic of order $109$, the inertial index is $e=12$, the exceptional multiplicity is $m=(|D|-1)/e=9$ and the Brauer tree $\sigma(\mathbf{B}_0)$ is of the following form:
   
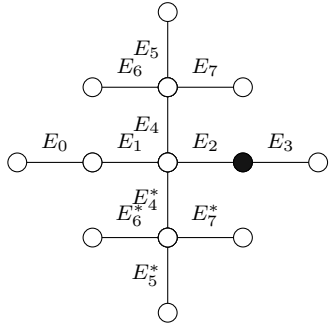
\[  
   \begin{tikzpicture}[solidnode/.style={circle, fill=black!92, inner sep=0pt, minimum size=2.5mm},hollownode/.style={circle, fill=white!84, inner sep=0pt, minimum size=2.5mm}, auto,bend left]
        \draw (0,0)  node[hollownode,draw]{} -- node[above] {\tiny $E_0$} (1,0) node[hollownode,draw]{};
        \draw (1,0)  node[hollownode,draw]{} -- node[above] {\tiny $E_1$} (2,0) node[hollownode,draw]{};
        \draw (2,0)  node[hollownode,draw]{} -- node[above] {\tiny $E_2$} (3,0) node[solidnode,draw]{};    
        \draw (3,0)  node[solidnode,draw]{} -- node[above] {\tiny $E_3$} (4,0) node[hollownode,draw]{};     
        \draw (2,0)  node[hollownode,draw]{} -- node[left] {\tiny $E_4$\!} (2,1) node[hollownode,draw]{};  
        \draw (2,0)  node[hollownode,draw]{} -- node[left] {\tiny $E_4^\ast$\!} (2,-1) node[hollownode,draw]{}; 
        \draw (2,1)  node[hollownode,draw]{} -- node[left] {\tiny $E_5$\!} (2,2) node[hollownode,draw]{};  
        \draw (2,-1)  node[hollownode,draw]{} -- node[left] {\tiny $E_5^\ast$\!} (2,-2) node[hollownode,draw]{};
        \draw (2,1)  node[hollownode,draw]{} -- node[above] {\tiny $E_6$} (1,1) node[hollownode,draw]{}; 
        \draw (2,1)  node[hollownode,draw]{} -- node[above] {\tiny $E_7$} (3,1) node[hollownode,draw]{};
        \draw (2,-1)  node[hollownode,draw]{} -- node[above] {\tiny $E_6^\ast$} (1,-1) node[hollownode,draw]{}; 
        \draw (2,-1)  node[hollownode,draw]{} -- node[above] {\tiny $E_7^\ast$} (3,-1) node[hollownode,draw]{};     
    \end{tikzpicture}  
    \]
In particular, the real stem consists of  4 edges, which we label $E_0,\ldots, E_3$, i.e. $b=4$.  Note that\/ $E_0$ is the trivial module.

By Corollary~\ref{cor:numberselfduals} there are $|D|-1=108$ isomorphism classes of non-projective self-dual indecomposable $\mathbf{B}_0$-modules, which by Theorem~\ref{thm:description_selfduals} are as given below:

\begin{enumerate}[\rm (i)]
        \item the $b=4$ simple modules $E_0$, $E_1$, $E_2$ and\/ $E_3$\,;
        \item the modules given by the $\kappa=8$ paths
        \begin{enumerate}[\,]
                \item \(
                    \xymatrix@R=0.0000pt@C=30pt{
                    {\Circle}\ar@<0.0ex>[r]^{E_{4}}&
                    {\Circle}\ar@<0.0ex>[r]^{E_{4}^\ast}&
                    {\Circle}
                    },
                    \)
                \item \(
                    \xymatrix@R=0.0000pt@C=30pt{ {\Circle}\ar@<0.0ex>[r]^{E_{5}}&{\Circle}\ar@<0.0ex>[r]^{E_{4}}&
                    {\Circle}\ar@<0.0ex>[r]^{E_{4}^\ast}&
                    {\Circle}\ar@<0.0ex>[r]^{E_{5}^\ast}&
                    {\Circle}
                    },
                    \) 
                \item \(
                    \xymatrix@R=0.0000pt@C=30pt{ {\Circle}\ar@<0.0ex>[r]^{E_{6}}&{\Circle}\ar@<0.0ex>[r]^{E_{4}}&
                    {\Circle}\ar@<0.0ex>[r]^{E_{4}^\ast}&
                    {\Circle}\ar@<0.0ex>[r]^{E_{6}^\ast}&
                    {\Circle}
                    },
                    \)  
                \item \(
                    \xymatrix@R=0.0000pt@C=30pt{ {\Circle}\ar@<0.0ex>[r]^{E_{7}}&{\Circle}\ar@<0.0ex>[r]^{E_{4}}&
                    {\Circle}\ar@<0.0ex>[r]^{E_{4}^\ast}&
                    {\Circle}\ar@<0.0ex>[r]^{E_{7}^\ast}&
                    {\Circle}
                    },
                    \) 
                \item \(
                    \xymatrix@R=0.0000pt@C=30pt{
                    {\Circle}\ar@<0.0ex>[r]^{E_{4}^\ast}&
                    {\Circle}\ar@<0.0ex>[r]^{E_{4}}&
                    {\Circle}
                    },
                    \)  
                \item \(
                    \xymatrix@R=0.0000pt@C=30pt{ {\Circle}\ar@<0.0ex>[r]^{E_{5}^\ast}&{\Circle}\ar@<0.0ex>[r]^{E_{4}^\ast}&
                    {\Circle}\ar@<0.0ex>[r]^{E_{4}}&
                    {\Circle}\ar@<0.0ex>[r]^{E_{5}}&
                    {\Circle}
                    },
                    \)  
                \item \(
                    \xymatrix@R=0.0000pt@C=30pt{ {\Circle}\ar@<0.0ex>[r]^{E_{6}^\ast}&{\Circle}\ar@<0.0ex>[r]^{E_{4}^\ast}&
                    {\Circle}\ar@<0.0ex>[r]^{E_{4}}&
                    {\Circle}\ar@<0.0ex>[r]^{E_{6}}&
                    {\Circle}
                    },
                    \) 
                \item \(
                    \xymatrix@R=0.0000pt@C=30pt{ {\Circle}\ar@<0.0ex>[r]^{E_{7}^\ast}&{\Circle}\ar@<0.0ex>[r]^{E_{4}^\ast}&
                    {\Circle}\ar@<0.0ex>[r]^{E_{4}}&
                    {\Circle}\ar@<0.0ex>[r]^{E_{7}}&
                    {\Circle}
                    },   
                    \)          
              \end{enumerate}
        all of which with direction $(1,-1)$ and multiplicity  \smallskip   $\mu=0$\,;   
        \item the modules given by the $\kappa=8$ paths\\
                    \( 
                
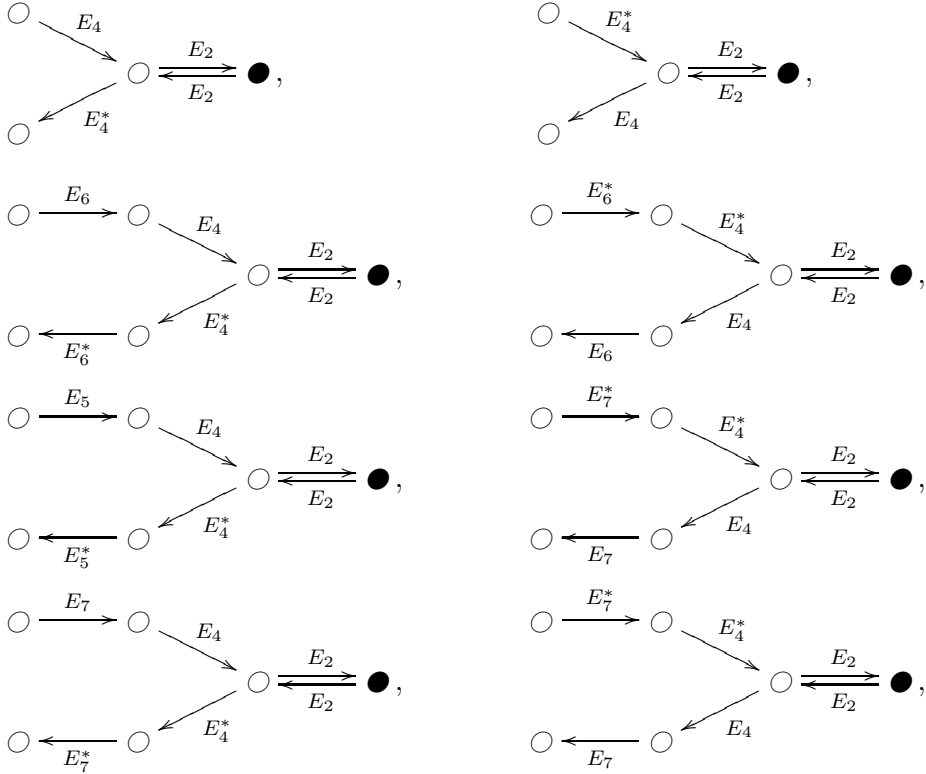
\xymatrix@R=0.0000pt@C=29pt{
 	                  & &\\
	                   {\Circle}\ar[ddr]^{E_{4}} &  &  \\
		            & & \\
		              &{\Circle}\ar[ddl]^{\:\:E_{4}^\ast}\ar@<0.3ex>[r]^{E_{2}} &    {\CIRCLE}\ar@<0.3ex>[l]^{E_{2}}\,,\\
		              & &\\
		            {\Circle} & & & \\
		              & &
	                }
            \) \qquad\qquad
            \( 
                \xymatrix@R=0.0000pt@C=29pt{
 	                  & &\\
	                   {\Circle}\ar[ddr]^{E_{4}^\ast} &  &  \\
		            & & \\
		              &{\Circle}\ar[ddl]^{\:\:E_{4}}\ar@<0.3ex>[r]^{E_{2}} &    {\CIRCLE}\ar@<0.3ex>[l]^{E_{2}}\,, \\
		              & &\\
		            {\Circle} & & & \\
		              & &
	                }
            \)\\ 
             \( 
                \xymatrix@R=0.0000pt@C=29pt{
 	                  && &\\
	                   {\Circle}\ar@<0.0ex>[r]^{E_{6}}&{\Circle}\ar[ddr]^{E_{4}} &  &  \\
		              && & \\
		              &&{\Circle}\ar[ddl]^{\:\:E_{4}^\ast}\ar@<0.3ex>[r]^{E_{2}} &    {\CIRCLE}\ar@<0.3ex>[l]^{E_{2}}\,,\\
		              && &\\
		              {\Circle}& {\Circle} \ar@<0.0ex>[l]^{E_{6}^\ast}& & & \\
		              && &
	                }
            \) 
            \( 
                \xymatrix@R=0.0000pt@C=29pt{
 	                  && &\\
	                   {\Circle}\ar@<0.0ex>[r]^{E_{6}^\ast}&{\Circle}\ar[ddr]^{E_{4}^\ast} &  &  \\
		              && & \\
		              &&{\Circle}\ar[ddl]^{\:\:E_{4}}\ar@<0.3ex>[r]^{E_{2}} &    {\CIRCLE}\ar@<0.3ex>[l]^{E_{2}}\,,\\
		              && &\\
		              {\Circle}& {\Circle} \ar@<0.0ex>[l]^{E_{6}}& & & \\
		              && &
	                }
            \) \\
            \( 
                \xymatrix@R=0.0000pt@C=29pt{
 	                  && &\\
	                   {\Circle}\ar@<0.0ex>[r]^{E_{5}}&{\Circle}\ar[ddr]^{E_{4}} &  &  \\
		              && & \\
		              &&{\Circle}\ar[ddl]^{\:\:E_{4}^\ast}\ar@<0.3ex>[r]^{E_{2}} &    {\CIRCLE}\ar@<0.3ex>[l]^{E_{2}}\,,\\
		              && &\\
		              {\Circle}& {\Circle} \ar@<0.0ex>[l]^{E_{5}^\ast}& & & \\
		              && &
	                }
            \) 
            \( 
                \xymatrix@R=0.0000pt@C=29pt{
 	                  && &\\
	                   {\Circle}\ar@<0.0ex>[r]^{E_{7}^\ast}&{\Circle}\ar[ddr]^{E_{4}^\ast} &  &  \\
		              && & \\
		              &&{\Circle}\ar[ddl]^{\:\:E_{4}}\ar@<0.3ex>[r]^{E_{2}} &    {\CIRCLE}\ar@<0.3ex>[l]^{E_{2}}\,,\\
		              && &\\
		              {\Circle}& {\Circle} \ar@<0.0ex>[l]^{E_{7}}& & & \\
		              && &
	                }
            \) \\ 
            \( 
                \xymatrix@R=0.0000pt@C=29pt{
 	                  && &\\
	                   {\Circle}\ar@<0.0ex>[r]^{E_{7}}&{\Circle}\ar[ddr]^{E_{4}} &  &  \\
		              && & \\
		              &&{\Circle}\ar[ddl]^{\:\:E_{4}^\ast}\ar@<0.3ex>[r]^{E_{2}} &    {\CIRCLE}\ar@<0.3ex>[l]^{E_{2}}\,,\\
		              && &\\
		              {\Circle}& {\Circle} \ar@<0.0ex>[l]^{E_{7}^\ast}& & & \\
		              && &
	                }
            \)    
            \( 
                \xymatrix@R=0.0000pt@C=29pt{
 	                  && &\\
	                   {\Circle}\ar@<0.0ex>[r]^{E_{7}^\ast}&{\Circle}\ar[ddr]^{E_{4}^\ast} &  &  \\
		              && & \\
		              &&{\Circle}\ar[ddl]^{\:\:E_{4}}\ar@<0.3ex>[r]^{E_{2}} &    {\CIRCLE}\ar@<0.3ex>[l]^{E_{2}}\,,\\
		              && &\\
		              {\Circle}& {\Circle} \ar@<0.0ex>[l]^{E_{7}}& & & \\
		              && &
	                }
            \) \\     
        with direction $(1,-1)$ and multiplicity $2\leq \mu\leq 9$\,;  
        \item the modules given by the $b=4$ paths 
            \begin{enumerate}[\,]
            \item \(
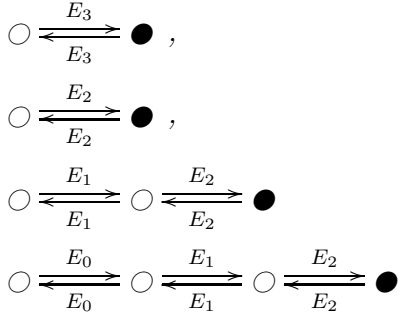
\xymatrix@R=0.0000pt@C=30pt{
                  {\Circle}\ar@<0.3ex>[r]^{E_3}&{\CIRCLE}\ar@<0.3ex>[l]^{E_3}
                   }
                  \)\,,
            \item \(
                    \xymatrix@R=0.0000pt@C=30pt{
                  {\Circle}\ar@<0.3ex>[r]^{E_2}&{\CIRCLE}\ar@<0.3ex>[l]^{E_2}
                   }
                  \)\,,
            \item \( \xymatrix@R=0.0000pt@C=30pt{
                   {\Circle}\ar@<0.3ex>[r]^{E_{1}} & {\Circle}\ar@<0.3ex>[l]^{E_1}\ar@<0.3ex>[r]^{E_{2}}&
                   {\CIRCLE}\ar@<0.3ex>[l]^{E_2}
                   }
                 \)  
            \item \(
            \xymatrix@R=0.0000pt@C=30pt{
                   {\Circle}\ar@<0.3ex>[r]^{E_{0}} & {\Circle}\ar@<0.3ex>[l]^{E_0}\ar@<0.3ex>[r]^{E_{1}}& {\Circle}\ar@<0.3ex>[l]^{E_1}\ar@<0.3ex>[r]^{E_{2}}&
                   {\CIRCLE}\ar@<0.3ex>[l]^{E_2}
                   }
                  \)
           \end{enumerate}
    with direction $(1,-1)$ and multiplicity \smallskip $2\leq \mu\leq m=9$.
\end{enumerate}
We note that {\rm(i)} yields $4$ isomorphism classes of self-dual modules, {\rm(ii)} yields $8$ isomorphism classes of self-dual modules, {\rm(iii)} yields $8\times 8$ isomorphism classes of self-dual modules and {\rm(iv)} yields $4\times8$ isomorphism classes of self-dual modules. Adding up, we see that we have described the $108$ required isomorphism classes of non-projective self-dual indecomposable \medskip  $\mathbf{B}_0$-modules.
\par
Next we describe the  positive hooks, according to the definition we gave in Remark \ref{rem:using_BC}(2). As the trivial character is the character afforded by the trivial module, it is positive on any generator $u$ of\/ $D_1$, so we obtain positive and negative signs for the values $\chi(u)$ for $\chi\in\mbox{\textup{Irr}}(\mathbf{B}_0)$ as follows:
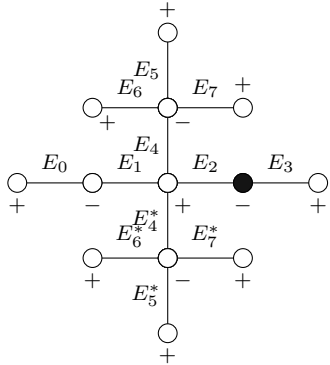
\[  
   \begin{tikzpicture}[solidnode/.style={circle, fill=black!92, inner sep=0pt, minimum size=2.5mm},hollownode/.style={circle, fill=white!84, inner sep=0pt, minimum size=2.5mm}, auto,bend left]
        \draw (0,0)  node[hollownode,draw]{} -- node[above] {\tiny $E_0$} (1,0) node[hollownode,draw]{};
        \draw (1,0)  node[hollownode,draw]{} -- node[above] {\tiny $E_1$} (2,0) node[hollownode,draw]{};
        \draw (2,0)  node[hollownode,draw]{} -- node[above] {\tiny $E_2$} (3,0) node[solidnode,draw]{};    
        \draw (3,0)  node[solidnode,draw]{} -- node[above] {\tiny $E_3$} (4,0) node[hollownode,draw]{};     
        \draw (2,0)  node[hollownode,draw]{} -- node[left] {\tiny $E_4$\!} (2,1) node[hollownode,draw]{};  
        \draw (2,0)  node[hollownode,draw]{} -- node[left] {\tiny $E_4^\ast$\!} (2,-1) node[hollownode,draw]{}; 
        \draw (2,1)  node[hollownode,draw]{} -- node[left] {\tiny $E_5$\!} (2,2) node[hollownode,draw]{};  
        \draw (2,-1)  node[hollownode,draw]{} -- node[left] {\tiny $E_5^\ast$\!} (2,-2) node[hollownode,draw]{};
        \draw (2,1)  node[hollownode,draw]{} -- node[above] {\tiny $E_6$} (1,1) node[hollownode,draw]{}; 
        \draw (2,1)  node[hollownode,draw]{} -- node[above] {\tiny $E_7$} (3,1) node[hollownode,draw]{};
        \draw (2,-1)  node[hollownode,draw]{} -- node[above] {\tiny $E_6^\ast$} (1,-1) node[hollownode,draw]{}; 
        \draw (2,-1)  node[hollownode,draw]{} -- node[above] {\tiny $E_7^\ast$} (3,-1) node[hollownode,draw]{};  
        \node at (0,-0.3) {\tiny $+$};
        \node at (1,-0.3) {\tiny $-$};
        \node at (2.2,-0.3) {\tiny $+$};
        \node at (3,-0.3) {\tiny $-$};
        \node at (4,-0.3) {\tiny $+$};
        \node at (1,-1.3) {\tiny $+$};
        \node at (2.2,-1.3) {\tiny $-$};
        \node at (3,-1.3) {\tiny $+$};
        \node at (2,-2.3) {\tiny $+$};
        \node at (2,2.3) {\tiny $+$};
        \node at (3,1.3) {\tiny $+$};
        \node at (2.2,0.8) {\tiny $-$};
        \node at (1.2,0.8) {\tiny $+$};
    \end{tikzpicture}  
    \]
    It follows from Remark \ref{rem:using_BC} that the positive hooks are the following 12 uniserial modules: $E_0$, $E_5$, $E_5^\ast$, $E_6$, $E_6^\ast$, $E_7$, $E_7^\ast$, $E_3$,\\
    \(
    \xymatrix@R=0.0000pt@C=30pt{
                    {\Circle}\ar@<0.0ex>[r]^{E_{1}}&
                    {\Circle}\ar@<0.0ex>[r]^{E_{4}}&
                    {\Circle}
                    }
    \) 
    with direction $(1,-1)$ and multiplicity $\mu=0$,\\
        \(
    \xymatrix@R=0.0000pt@C=30pt{
                    {\Circle}\ar@<0.0ex>[r]^{E_{4}}&
                    {\Circle}\ar@<0.0ex>[r]^{E_{2}}&
                    {\Circle}
                    }
    \) 
    with direction $(1,-1)$ and multiplicity $\mu=0$,\\
        \(
    \xymatrix@R=0.0000pt@C=30pt{
                    {\Circle}\ar@<0.0ex>[r]^{E_{2}}&
                    {\Circle}\ar@<0.0ex>[r]^{E_{4}^\ast}&
                    {\Circle}
                    }
    \) 
    with direction $(1,-1)$ and multiplicity $\mu=0$, and \\
        \(
    \xymatrix@R=0.0000pt@C=30pt{
                    {\Circle}\ar@<0.0ex>[r]^{E_{4}^\ast}&
                    {\Circle}\ar@<0.0ex>[r]^{E_{1}}&
                    {\Circle}
                    }
    \) 
    with direction $(1,-1)$ and \medskip multiplicity $\mu=0$. 
\par    
The Brauer correspondent of\/ $\mathbf{B}_0$ is the principal block of\/ $N_G(D_1)$, so in this case, we actually know from Theorem~\ref{thm:type_selfduals_general_case}(a) that all self-dual modules have the same type, namely the type of the trivial module. 

However, we illustrate the procedure of Theorem \ref{thm:type_selfduals_general_case} to determine the type of the module~$M$ given by the path 
  
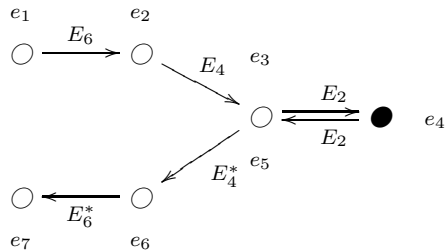
\[ 
        \xymatrix@R=0.0000pt@C=29pt{
 	                  && &\\
 	                  {\text{\tiny $e_1$}}&{\text{\tiny $e_2$}}& &\\
	                   {\Circle}\ar@<0.0ex>[r]^{E_{6}}&{\Circle}\ar[ddr]^{E_{4}} & {\text{\tiny $e_3$}}  &  \\
		              && & \\
		              &&{\Circle}\ar[ddl]^{\:\:E_{4}^\ast}\ar@<0.3ex>[r]^{E_{2}} &    {\CIRCLE}{\quad\text{\tiny $e_4$}}\ar@<0.3ex>[l]^{E_{2}}\\
		              && {\text{\tiny $e_5$}} &\\
		              {\Circle}& {\Circle} \ar@<0.0ex>[l]^{E_{6}^\ast}& & & \\
		             {\text{\tiny $e_7$}} & {\text{\tiny $e_6$}}& & \\
		              && &
	       }
    \]
with direction $(1,-1)$ and multiplicity  $\mu=6$  assuming we don't know, a priori, that the type is the type of the trivial module. Notice that we have named the vertices $e_1$ to $e_7$, as the path is of length $6$. We apply \cite[Theorem 3.5]{BC} in order to obtain the position of\/ $M$ in $\Gamma_s(\mathbf{B}_0)$. More precisely, in order to use the formula of \cite[Theorem 3.5]{BC} for $d^+(M,H_M^+)$, we need to determine a series of local parameters defined in \cite{BC}. These are: $\alpha=1$, $s=6$, $v_a=e_6=v_1$, $v_z=e_1=v_{n+1}$, $S_a=E_6^\ast=X_1$, $S_z=E_6=X_n$, $X_{i_1}=E_6^\ast$,  $X_{i_2}=E_4^\ast$,  $X_{i_3}=E_2$, $X_{i_4}=E_2$, $X_{i_5}=E_4$, $X_{i_6}=E_6$, $n=33$, $k_0=4$, $\eta=\mu-2=6-2=4$ and the distance of\/ $M$ is caclulated to the positive hook\/ $H_M^+=E_6$. This yields:
\[
d^+(M,H_M^+)=(n-1)/2+\eta\cdot e= (33-1)/2+ 4\cdot 12 = 16+48=64\,.
\]
So the integer $i$ of Theorem \ref{thm:type_selfduals_general_case}(a) is $i=32$ and the type of\/ $M$ is the type of the self-dual hook\/ $\Omega^{-64}(H_M^+)\cong\Omega^{-4}(H_M^+)\cong E_0\cong k$ (where the last-but-one isomorphism follows from the shape of the Brauer tree). 
\end{Example}
\bigskip

\section*{Acknowledgment}
The idea for this paper arose during the conference on Representation Theory which took place at MFO Oberwolfach in April 2023. Much of the detail was then worked out during a visit of the second author to the University of Kaiserslautern-Landau (RPTU Kaiserslautern-Landau) in July 2023. We would like to thank both institutions and the DFG/SFB-TRR 195 for their support.


\bigskip


\appendix
\section{Parametrisation of the indecomposable modules}\label{app:A}%

In this appendix, we shortly recall how the indecomposable modules of a block\/ $\mathbf{B}$ with a cyclic defect group $D\cong C_{p^a}$ with $a\geq 1$ and inertial index $e$ can be described using the data given by its plannar embedded Brauer tree $\sigma(\mathbf{B})$ with exceptional multiplicity $m$. If not otherwise stated, we keep the notation introduced in Section \ref{sec:intro} and Section \ref{sec:prelim}. In particular, we let\/ $\{E_0,...,E_{e-1}\}$ denote the set of all pairwise non-isomorphic  irreducible $\mathbf{B}$-modules.  We refer the reader to \cite{Ja}, and also \cite{HN} for details.

\begin{Notation}\label{nota:paths_dir_mult}
Based on standard results of Janusz~\cite[\S5]{Ja} and work of Bleher-Chinburg \cite{BC}, up to isomorphism,  each non-projective non-irreducible indecomposable $\mathbf{B}$-module $X$  can be encoded  using the following three  parameters.
    \begin{enumerate}[\rm(1)]
        \item  A \emph{path} on $\sigma(\mathbf{B})$, which is by definition  a directed connected subgraph of
        $\sigma(\mathbf{B})$ of one of the following two \medskip types: \\
         \(  
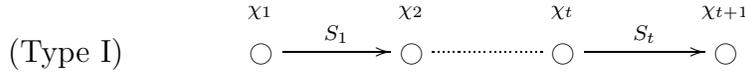
\xymatrix@R=0.0000pt@C=40pt{
         & {_{\chi_1}}
         & {_{\chi_2}}
         & {_{\chi_{t}}}
         & {_{\chi_{t+1}}}\\
         \text{\!(Type I)} & 
         {\Circle}\ar@<0.0ex>[r]^{S_{1}}
         &{\Circle}\ar@{.}[r]
         &{\Circle}\ar@<0.0ex>[r]^{S_{t}}
         &{\Circle}   \\
         & & & \phantom{X} &
        }
        \)\\
        where $t\geq 1$ and\/ $S_1,\ldots,S_t\in\{E_0,\ldots,E_{e-1}\}$ are pairwise distinct and  not adjacent to the exceptional vertex, 
        \medskip or, if\/ $m\geq 2$,\\
        (Type II)
        
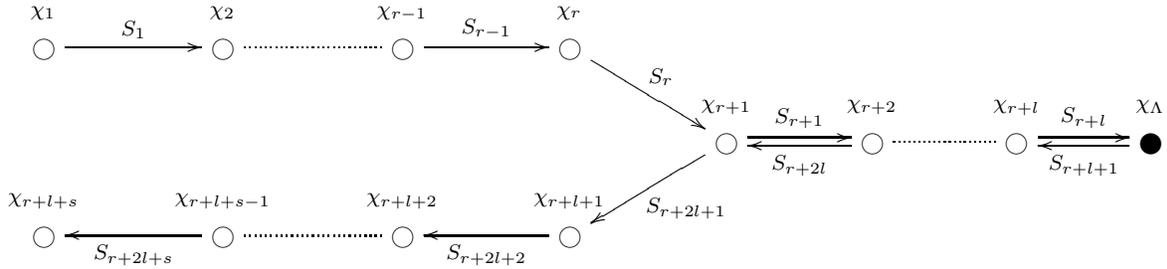
\[ \xymatrix@R=0.0000pt@C=29pt{
 	    {_{\chi_1}}& {_{\chi_2}}&{_{\chi_{r-1}}}&{_{\chi_{r}}}&& &\\
	    {\Circle}\ar@<0.0ex>[r]^{S_{1}}&{\Circle}\ar@{.}[r]&{\Circle}\ar@<0.0ex>[r]^{S_{r-1}}&{\Circle}\ar[ddr]^{S_{r}} & & &  \\
		&&&&{_{\chi_{r+1}}}&{_{\chi_{r+2}}}&{_{\chi_{r+l}}}&{_{\chi_{\Lambda}}} \\
		&&&&{\Circle}\ar[ddl]^{\:\:S_{r+2l+1}} \ar@<0.3ex>[r]^{S_{r+1}}&                         {\Circle}\ar@<0.3ex>[l]^{S_{r+2l}}\ar@{.}[r]&                    {\Circle}\ar@<0.3ex>[r]^{S_{r+l}}&                              {\CIRCLE}\ar@<0.3ex>[l]^{S_{r+l+1}}\\
		{_{\chi_{r+l+s}}}&{_{\chi_{r+l+s-1}}}&{_{\chi_{r+l+2}}}&{_{\chi_{r+l+1}}}&& &\\
		{\Circle}&{\Circle}\ar@<0.0ex>[l]^{S_{r+2l+s}}\ar@{.}[r]&{\Circle}& {\Circle} \ar@<0.0ex>[l]^{S_{r+2l+2}}& & & \\
		&&&&& &
	   }
         \] 
        where $l,r,s\geq 0$, $t:=r+2l+s\geq 2$, possibly $\chi_{r+1}=\chi_{\Lambda}$ (in which case $l:=0$),  $S_1,\ldots,S_{r+l},S_{},\ldots,S_t\in\{E_0,\ldots,E_{e-1}\}$ are pairwise distinct, 
        $S_{r+l+i}\cong S_{r+l-(i-1)}$ for each $1\leq i\leq l$ and\/ $S_{r+l}$ is adjacent to the exceptional vertex. \\
        These paths may be seen as an ordered sequence $(S_1,\ldots,S_{t})$ of edges of  $\sigma(\mathbf{B})$, called \emph{top-socle sequence} of\/ $X$, and  where $S_i,S_{i+1}$ have a common vertex for every $1\leq i\leq t$,  the odd-labelled edges are constituents of the head of\/ $X$ and the even-labelled edges are constituents of the socle of\/ $X$, or conversely.
        \item A  \emph{direction} $\varepsilon=(\varepsilon_1,\varepsilon_{t})$, where for $i\in\{1,t\}$  we set\/ $\varepsilon_i=1$ if\/ $S_i$ is in the head of\/ $X$ and\/ $\varepsilon_i=-1$ if\/ $S_i$ is in the socle of\/ $X$.
        \item A \emph{multiplicity} $\mu$.  If\/ $m=1$, then $\mu:=0$. If\/ $m>1$, then $\mu$ corresponds to the number of times that an irreducible module $S_{j}$ adjacent to the exceptional vertex  occurs as a composition factor of~$X$ (which is independent of the choice of\/ $S_j$).  
    \end{enumerate}
The structure of an indecomposable $\mathbf{B}$-module $X$ with top-socle sequence $(S_1,\ldots,S_{t})$, direction $\varepsilon$ and multiplicity $\mu$ can then be understood as follows. 
First, write $\Delta:={\{1\leq i\leq t \mid S_{i}\in \hd(X)\}}$. Then, for each $i\in\Delta$ there is a submodule $X_i$ of\/ $X$ with unique head composition factor $S_i$ and with socle composition factors consisting of the edges of\/ $\sigma(\mathbf{B})$ adjacent to $S_i$.
    \[  
    X_i\,=\quad
    \begin{tikzcd}[row sep=large, column sep=small]
       & S_i  \arrow[ld, no head, no head, dotted] \arrow[rd, no head, no head, dotted]   &  \\
     S_{i-1}  &     & S_{i+1}
    \end{tikzcd}
    \]  

If\/ $S_{i-1}$ and\/ $S_i$ have $\chi^{}_i$ as  common vertex and\/ $S_{i}$ and\/ $S_{i+1}$ have $\chi^{}_{i+1}$ as  common vertex, then the composition factors of the left leg of\/ $X_i$, from bottom to top, correspond to
the edges of\/ $\sigma(\mathbf{B})$ adjacent to $\chi_{i}$ encountered on a clockwise walk around\/ $\chi_{i}$ from $S_{i-1}$ to $S_i$. The composition factors of the right leg of\/ $X_i$, from top to bottom, correspond to the edges of\/ $\sigma(\mathbf{B})$ adjacent to $\chi_{i+1}$ encountered on
a counter-clockwise walk around\/ $\chi_{i+1}$ from $S_i$ to $S_{i+1}$. 
If the path starts with $S_i$, resp. ends with $S_i$, then $X_i$ has no left leg, resp. no right leg. In other words, $X_i$ is uniserial.

If the vertex between $S_{i-1}$ and\/ $S_i$, resp. between $S_{i}$ and\/ $S_{i+1}$ is the exceptional vertex, then we walk around\/ $\chi_{i}$, resp. $\chi_{i+1}$, several full circles in such a way that\/ $S_i$ occurs with multiplicity $\mu$ in the left, resp. right, leg of\/ $X_i$. Finally, the indecomposable module $X$ is obtained abstractly by amalgamating the modules isomorphic to $X_i$ along their common socle constituents, i.e. $X =\sum_{i\in\Delta} X_i$. It can be visualised as follows:
    \[  
    X\,=\quad
    \begin{tikzcd}[row sep=large, column sep=small]
             \cdots &  S_{i-2}\arrow[rd, no head, no head, dotted] &    & S_i  \arrow[ld, no head, no head, dotted] \arrow[rd, no head, no head, dotted]   &  & S_{i+2}  \arrow[ld, no head, no head, dotted] \arrow[rd, no head, no head, dotted]  & & \\
                    &   &  S_{i-1} &    &  S_{i+1}  & & S_{i+3} &\cdots
    \end{tikzcd}
    \]  
\end{Notation}
\bigskip

\begin{Rem}~\vspace*{-1mm} \label{rem:uniserials}
    \begin{enumerate}[\rm(a)]
        \item When $m=1$, even if one vertex is artificially designated to be exceptional, the above parametrization considers that there is no exceptional vertex. Thus, in this case, for simplicity we draw only paths of Type I. 
        \item     Warning: the set of indecomposable modules obtained by combining all possible paths, directions and multiplicities are not pairwise non-isomorphic. Even-length paths can give rise to doubles when the direction is switched. Such configurations are called \emph{mirror images} in \cite{BC}. This is the reason why it is helpful to consider directions.
        \item In particular, when considering uniserial modules, we may always assume that the direction is $(1,-1)$. Uniserial modules with direction $(-1,1)$ are then doubles of those with direction $(1,-1)$.
        \item A non-projective non-irreducible uniserial $\mathbf{B}$-module can be encoded by five  types of paths:
    \begin{enumerate}[\rm(U1)]
        \item    \(
                    \xymatrix@R=0.0000pt@C=30pt{
                    {\Circle}\ar@<0.0ex>[r]^{S_{1}}&{\Circle}\ar@<0.0ex>[r]^{S_{2}}&{\Circle}
                    }
                \)    with direction $(1,-1)$ and multiplicity $\mu=0$;
        \item   \(
                    \xymatrix@R=0.0000pt@C=30pt{
                    {\Circle}\ar@<0.0ex>[r]^{S_{1}}&{\Circle}\ar@<0.0ex>[r]^{S_{2}}&{\CIRCLE}
                    }
                \) with direction $(1,-1)$ and multiplicity $\mu=1$;
        \item   \(
                    \xymatrix@R=0.0000pt@C=30pt{
                    {\CIRCLE}\ar@<0.0ex>[r]^{S_{1}}&{\Circle}\ar@<0.0ex>[r]^{S_{2}}&{\Circle}
                    }
                \) with direction $(1,-1)$  and multiplicity $\mu=1$;        
        \item   \(
                    \xymatrix@R=0.0000pt@C=30pt{
                    {\Circle}\ar@<0.0ex>[r]^{S_{1}}&{\CIRCLE}\ar@<0.0ex>[r]^{S_{2}}&{\Circle}
                    }
                \)  with direction $(1,-1)$ and multiplicity ${1\leq\mu\leq m}$; or
        \item   \(
               \xymatrix@R=0.0000pt@C=30pt{
                   {\Circle}\ar@<0.3ex>[r]^{S_{1}}&{\CIRCLE}\ar@<0.3ex>[l]^{S_1}
                   }
                \)  with direction $(1,-1)$ and multiplicity $2\leq\mu\leq m$.     
        \end{enumerate}
       Here it is implicitly assumed\/ $S_1,S_2\in\{E_0,\ldots,E_{e-1}\}$ and that paths of type {\rm(U1)} can occur for an $m\geq 1$ arbitrary, whereas paths of types {\rm(U2)}, {\rm(U3)}, {\rm(U4)} and {\rm(U5)} occur only when $m\geq 2$.
    \end{enumerate}
\end{Rem}

It is also easy to describe the self-dual uniserial modules using this notation. 

\begin{Lemma}\label{lem:uniserial_selfdual}
Let\/ $X$ be a non-projective non-irreducible uniserial $\mathbf{B}$-module. If\/ $X$ is self-dual, then it is parametrized by one of the following paths, where $S_1\in\{E_0,\ldots,E_{e-1}\}$:\\
if\/ $m\geq 1$, 
\begin{enumerate}[\rm(USD1)]
    \item  \(
                    \xymatrix@R=0.0000pt@C=30pt{
                    {\Circle}\ar@<0.0ex>[r]^{S_{1}}&{\Circle}\ar@<0.0ex>[r]^{S_{1}^\ast}&{\Circle}
                    }
                \)    with direction $(1,-1)$ and multiplicity $\mu=0$;
\end{enumerate}
or, if\/ $m\geq 2$,
\begin{enumerate}[\rm(USD2)]
    \item  \(
                \xymatrix@R=0.0000pt@C=30pt{
                    {\Circle}\ar@<0.0ex>[r]^{S_{1}}&{\CIRCLE}\ar@<0.0ex>[r]^{S_{1}^{\ast}}&{\Circle}
                    }
                \)  with direction $(1,-1)$ and multiplicity ${1\leq\mu\leq m}$; or
\end{enumerate}                
\begin{enumerate}[\rm(USD3)]
    \item \(
               \xymatrix@R=0.0000pt@C=30pt{
                   {\Circle}\ar@<0.3ex>[r]^{S_{1}}&{\CIRCLE}\ar@<0.3ex>[l]^{S_1}
                   }
                \)  with direction $(1,-1)$, multiplicity $2\leq\mu\leq m$ and\/ $S_1\cong S_1^{\ast}$.
    
\end{enumerate}
\end{Lemma}

\begin{proof}
By Remark~\ref{rem:uniserials}(c) the non-irreducible uniserial $\mathbf{B}$-modules are given by paths of type (U1) to (U5). 
Now, if a uniserial module is self-dual, then we obtain from Formula (\ref{form:socle_head}) that\/ $S_2\cong S_1^\ast$ for modules of type (U1), (U2), (U3) and (U4), whereas $S_1\cong S_1^\ast$ for modules of type (U5). As the exceptional vertex is always on the real stem and duality induces a reflection with respect to the real stem, it follows from the above that modules of type (U2) and (U3) cannot be self-dual. The remaining modules of type (U1), (U4) and (U5) with direction and multiplicity as given in the claim are all self-dual by the description of their composition factors in Notation~\ref{nota:paths_dir_mult} and Formula~(\ref{form:layers_dual}) for the socle layers of the dual.
\end{proof}


\begin{thebibliography}{HLb}

\bibitem[B]{B} D.~J.~Benson, {\em Representations and Cohomology. {I}}, Cambridge Studies in Advanced Mathematics 30, Cambridge University Press, Cambridge, 1991.

\bibitem[BC]{BC} F.~M.~Bleher, T.~Chinburg, {Locations of modules for Brauer tree algebras,} {\em J.~Pure Appl.~Algebra} {\bf 169} (2002) 109--135.


\bibitem[D]{D} E.~C.~Dade, {Blocks with cyclic defect groups}, {\em Ann. Math.} {\bf84} (1966) 20--48.

\bibitem[F]{F} W.~Feit, {\em The representation theory of finite groups,} North-Holland Math. Library 25, 1982.

\bibitem[GW]{GW} R.~Gow, W.~Willems, {A note on Green correspondence and forms,} {\em Comm. in Algebra} {\bf23} (1995) 1239--1248.

\bibitem[Gr]{Gr} J.~A.~Green, {Walking around the Brauer Tree}, {\em J. Austral. Math. Soc.} {\bf17} (1974) 197--213.

\bibitem[H]{H} G.~Hi{\ss}, {The Brauer trees of the Ree Groups,} {\em Comm.~Algebra} {\bf 19} (1991) 871--888.

\bibitem[HLa]{HLa} G.~Hi{\ss}, C. Lassueur, {The classification of the trivial source modules in blocks with cyclic defect groups,} {\em Algebr.~Represent. Theory} {\bf 24} (2021) 673--698.

\bibitem[HLb]{HLb} G.~Hi{\ss}, K.~Lux, {\em Brauer trees of sporadic groups,} Oxford Science Publications, The Clarendon Press, Oxford University Press, New York, 1989.

\bibitem[HN]{HN} G.~Hi{\ss}, N. Naehrig,  {The classification of the indecomposable liftable modules in blocks with cyclic defect groups,} {\em Bull. Lond. Math. Soc.} {\bf 44} (2012) 974--980.

\bibitem[Ja]{Ja} G.~T.~Janusz, {Indecomposable modules for finite groups}, {\em Ann. Math.} {\bf 89} (1969) 209--241.

\bibitem[M]{M} J.~C.~Murray, {Frobenius-Schur indicators of characters in blocks with cyclic defect}, {\em J. Algebra} {\bf533} (2019) 90--105.

\bibitem[We]{We} P.~Webb, {\em A course in finite group representation theory}, Cambridge Studies in Advanced Mathematics 161, Cambridge University Press, Cambridge, 2016.

\bibitem[W]{W} W.~Willems, {Duality and forms in representation theory,} {\em Representation theory of finite groups and finite-dimensional algebras} (Bielefeld, 1991), {Progr. Math.}, 95, Birkh\"auser, Basel, 1991, 509--520.

\end{thebibliography}
\end{document}